\documentclass[makeidx]{amsart} 
\usepackage{amsmath}
\usepackage{amssymb}
\usepackage{color}
\usepackage{epsfig}
\usepackage{hyperref}
\newtheorem{Proposition}{Proposition}
  \newtheorem{Remark}{Remark}
  
  \newtheorem{Lemma}[Proposition]{Lemma}
  
  \newtheorem{Theorem}{Theorem}
 
 \newtheorem{Definition}[Proposition]{Definition}
 \newtheorem{Note}[Remark]{Note}

\newcommand {\z}{{\noindent}}

\def\L|{\left\|}
\def\R|{\right\|}
\def\epsilon{\varepsilon}
\def\blackslug{\hbox{\hskip 1pt \vrule width 4pt height 8pt depth 1.5pt
\hskip 1pt}}
\def\qed{\quad\blackslug\lower 8.5pt\null\par}

 \def\RR{\mathbb{R}}

\makeindex

\author{O. Costin$^1$ and S. Tanveer$^1$} \address{$^1$ Mathematics Department\\The Ohio State University\\Columbus, OH 43210 } 

\title{Analytical approximation of  Blasius' similarity solution with rigorous error bounds}

\begin{document}
$ $ \vskip -0.2cm
\begin{abstract}
  We use a recently developed method \cite{Costinetal}, \cite{Dubrovin} 
   to find accurate analytic approximations with rigorous error bounds for the classic 
  similarity solution of Blasius of the boundary layer equation in fluid mechanics, the two point boundary value problem
   $f^{\prime \prime \prime} + f f^{\prime \prime} =0$ with $f(0)=f^\prime (0)=0$
  and $\lim_{x \rightarrow \infty} f^\prime (x) =1$. The approximation
  is given in terms of a polynomial in $[0, \frac{5}{2}]$ 
  and in terms of the error function
  in $[\frac{5}{2}, \infty)$. The two representations for the solution
  in different domains match at 
  $x=\frac{5}{2}$
  determining all free
  parameters in the problem, in particular 
  $f^{\prime \prime} (0) =0.469600 \pm 0.000022 $ 
  at the wall  The method can in principle
  provide approximations to any desired accuracy 
  for this or wide classes of linear or nonlinear differential equations with initial or
  boundary value conditions. The analysis relies on controlling
  the errors in the approximation through contraction 
 mapping arguments, using energy bounds for the Green's function of the linearized problem.
  \end{abstract}

\vskip -2cm
\maketitle

\today
%\tableofcontents

\section{Introduction and main results}

Finding exact expressions for solutions to nonlinear problems is an
  important area of research. Closed
  form solutions however exist only for a small sub-class of problems
(essentially for integrable models). On the other-hand, if a
  problem involves some small parameter $\epsilon$ (or a large
  parameter) and the limiting problem is exactly solvable, then there
  exist quite general asymptotic methods to obtain convenient
  expansions for the perturbed problem.

Indeed, consider for instance the question of finding the solution to
 $\mathcal{N} [ u ,\epsilon]=0$, where $\mathcal{N}$ is a (possibly nonlinear) differential operator 
in some space of functions satisfying boundary/initial
conditions and that $u_0$ is the solution at $\epsilon=0$. Existence and uniqueness of a solution
$u$ as well bounds on the error $E=u-u_0$ 
may be found as follows. We write
$$ L\,E=- \delta-\mathcal{N}_1(E)$$ 
where $L=\frac{\partial\mathcal{N}}{\partial u}|_{u=u_0}$, $\delta = \mathcal{N} [u_0]$ and 
$\mathcal{N}_1(E)=\mathcal{N}(u_0+ E)-L\,E  =O(\epsilon^2)$, inverting  $L$ in a suitable way, subject to the initial/boundary
conditions, and using the contractive mapping theorem in an adapted norm to control the small nonlinearity $\mathcal{N}_1(E)$.

A relatively general strategy has been recently
been employed \cite{Costinetal}-\cite{Dubrovin} in problems without explicit small parameters. The approach 
uses exponential asymptotic methods and classical orthogonal polynomial techniques  
to find a function $u_0$ which is a very accurate global approximation of the sought solution $u$, 
in the sense that 
$\mathcal{N}(u_0)$ is very small in a suitable norm and the boundary conditions are satisfied up to small errors. 
Once this is accomplished, a 
 perturbative approach similar with the one above  applies 
with the role of $\epsilon$   played by the norm of $\delta $ and one obtains  
an actual solution $u$ by controlling the equation satisfied by $E=u-u_0$.

In the present paper,
we apply this strategy to the well-known Blasius similarity equation arising in boundary
layer fluid-flow past a flat plate.
We 
improve the methods in \cite{Costinetal}-\cite{Dubrovin} 
by replacing the laborious detailed  analysis  of the Green's function  with softer  energy methods.

The Blasius similarity solution 
satisfies
\begin{equation}
\label{1}
f^{\prime \prime \prime} (x) + f (x) f^{\prime \prime} (x) = 0  ~~ \, {\rm for} ~x \in (0, \infty) 
\end{equation}
with boundary conditions
\begin{equation}
\label{2}
f(0) = 0 \ , f^\prime (0) = 0 \ ,   \lim_{x \rightarrow +\infty} f^\prime (x) = 1 
\end{equation}
This equation and various modifications
have garnered much attention since Blasius \cite{Blasius}
derived it\footnote{The equation
in the original Blasius' paper has a coefficient $\frac{1}{2}$ for $f f^{\prime}$;
however the
change of variable
$x \rightarrow x/\sqrt{2}$, $f \rightarrow f/\sqrt{2}$ transforms (\ref{1}) into Blasius' original equation. 
Thus, $f^{\prime \prime} (0) = 0.469600 \pm 0.000022$ 
transforms to $f^{\prime \prime} (0) = 0.3320574 \pm 0.000016$ 
in the original variables.}.   Existence and uniqueness were first proved 
by Weyl in \cite{Weyl}. Issues of existence and
uniqueness for this and related equations
have been considered as well by many 
authors (see for eg. \cite{Hussaini}, \cite{Brighi}, the latter being a review paper).
Hodograph transformations 
\cite{Callegari} allow a 
convergent power series representation in the entire domain,
but the convergence is slow at the edge of the domain and the
representation is not quite convenient in finding an approximation to $f$ directly.
Empirically, there has been quite a bit of interest in obtaining simple
expressions for Blasius and related similarity solutions. Liao \cite{Liao}
for instance
introduced a formal method for
an emprically accurate approximation; the theoretical
basis for this procedure and its limitations remain however unclear.
We are unaware of any rigorous error control for this or any
other efficient approximation in terms of simple functions.

In \cite{Weyl}, using a transformation introduced
by Topfer \cite{Topfer}, 
it is also proved that $f$ in \eqref{1}, \eqref{2} can be expressed as
\begin{equation}
\label{2.1}
f (x) = a^{-1/2} F \left ( a^{-1/2} x \right ),
\end{equation}
where
$F$ satisfies the initial value problem
\begin{equation}
\label{3}
F^{\prime \prime \prime} (x) + F (x) F^{\prime \prime} (x) = 0  ~~ \, {\rm for} ~x \in (0, \infty) 
\end{equation}
with initial conditions
\begin{equation}
\label{4}
F(0) = 0 \ , F^\prime (0) = 0 \ ,   F^{\prime \prime} (0) = 1, 
\end{equation}
In \eqref{2.1}, $\lim_{x \rightarrow \infty} F^\prime (x)=a\in\RR^+$ (cf. \cite{Weyl,Topfer}).
Conversely, if $f(x)$ satisfies \eqref{1}-\eqref{2} with $f^{\prime \prime} (0) = \beta > 0$ (physically, 
this corresponds to positive wall stress), it may be checked that
$F (x) = \beta^{-1/3} f (\beta^{-1/3} x ) $ satisfies \eqref{3}-\eqref{4}.
Because of this equivalence, 
it is more convenient to find an approximate analytical representation for the solution $F$
of (\ref{3})-(\ref{4}) and determine the value of $a $. The solution $f$ to the original problem is obtained
through the transformation (\ref{2.1}); 
the stress  at the wall is given by
\begin{equation}
\label{4.1}
f^{\prime \prime} (0) = a^{-3/2} 
\end{equation}
\subsection{Approximate representation of the solution}
Let
\begin{equation}
\label{6}
P(y) = \sum_{j=0}^{12} \frac{2}{5 (j+2)(j+3)(j+4)} p_j y^{j} \, 
\end{equation}
where $[p_0,...,p_{12}]$ are given by
{\small \begin{multline}
\Big[
-\frac{510}{10445149}, 
-\frac{18523}{5934}, -\frac{42998}{441819}, \frac{113448}{81151},-\frac{65173}{22093},\frac{390101}{6016},-\frac{2326169}{9858}, \\ \frac{4134879}{7249}, -\frac{1928001}{1960},
\frac{20880183}{19117} , -\frac{1572554}{2161} ,  \frac{1546782}{5833} , -\frac{1315241}{32239}\Big]
\end{multline}}
Define
\begin{equation}
\label{10.2}
t(x)=\frac{a}{2}(x+b/a)^2,
\ I_0 (t) = 1 - \sqrt{\pi t} e^{t} {\rm erfc} (\sqrt{t})  \ , ~~ 
J_0 (t) = 1 - \sqrt{2 \pi t} e^{2t} {\rm erfc} (\sqrt{2t}) \ ,
\end{equation}
where ${\rm erfc}$ denotes the complementary error function and let
\begin{equation}
\label{10.1}
q_0 (t) = 
2 c \sqrt{t} e^{-t} I_0  + 
c^2 e^{-2t} \left ( 2 J_0 - I_0 - I_0^2 \right ) ,
\end{equation}
The theorem below provides an accurate representation of $F$ of \eqref{3}, \eqref{4}.
\begin{Theorem}
\label{Thm1}Let $F_0$ be defined by
\begin{equation}
\label{6.0}
F_0 (x) =
\left\{ \begin{array}{cc}
\frac{x^2}{2} + x^4 P \left (\frac{2}{5} x \right ) \  \text{for $x\in [0,\frac52]$} \\
   a x + b + \sqrt{\frac{a}{2t (x)}} q_0 ( t(x)) \  \text{for $x > \frac52$}
\end{array}\right.
\end{equation}
Then, there is a unique triple $(a,b,c)$ close to $(a_0,b_0,c_0)= \left ( \frac{3221}{1946}, 
-\frac{2763}{1765}, \frac{377}{1613} \right )$ in the sense that $(a,b,c)\in\mathcal{S}$ where
\begin{equation}
\label{8}
\mathcal{S} = \left \{ (a, b, c)\in\RR^3: \sqrt{(a-a_0)^2+\frac{1}{4} (b-b_0)^2+
\frac{1}{4} (c-c_0)^2 } \le \rho_0 := 5 \times 10^{-5} \right \}
\end{equation}
with the property that $F_0$ is a representation of the actual solution $F$ 
to the initial value problem (\ref{3})-(\ref{4}) within small errors. More precisely,
\begin{equation}
\label{5}
F(x) = F_0 (x) + E (x)  \ ,
\end{equation}
where the error term $E$ satisfies
\begin{equation}
\label{errBoundsE}
\| E^{\prime \prime} \|_{\infty} \le 3.5 \times 10^{-6} \ ,  
\| {E}^{\prime} \|_{\infty} \le 4.5 \times 10^{-6}  \ ,  
\| {E} \|_{\infty} \le 4 \times 10^{-6}   \ \text{on $[0,\frac52]$}
\end{equation}
and for $x\ge \frac52$
\begin{multline}\label{eq15}
\Big |  {E}  \Big | 
\le 1.69 \times 10^{-5}  t^{-2} e^{-3 t}   \ ,    
\Big | \frac{d}{dx}  E  \Big | 
\le 9.20 \times 10^{-5} t^{-3/2} e^{-3 t}   \\
\Big | \frac{d^2}{dx^2}  E  \Big | 
\le 5.02 \times 10^{-4} t^{-1} e^{-3 t}   
\end{multline}
\end{Theorem}

\begin{Remark}{\rm 
Certainly, $F$ is smooth since it is is an actual 
solution of \eqref{3},\eqref{4}, which exists on $[0, \infty )$ 
and is unique, see \cite{Weyl}. 
However, the particular choice $(a, b, c) \in \mathcal{S}$ in Theorem \ref{Thm1} needed in order for 
$F=F_0+{E}$ to solve (\ref{3})-(\ref{4}) does not ensure continuity of
the approximate solution $F_0$ at $x=\frac{5}{2}$. Nonetheless, 
if  $F_0,F_0',F_0''$ are needed to be continuous, this can be achieved 
by a slightly different choice of $(a,b,c)\in \mathcal{S}$  (see Remark 
\ref{remCont}), namely
\begin{equation}
\label{eqabcCont}
(a, b, c) =  
\left (1.6551904561499...,-1.565439826457..,0.233728727537...
\right ). 
\end{equation}
}

\rm Note also that \eqref{eq15} implies 
not only small absolute errors 
(that, in the far field hold even for the approximation of $F^{\prime \prime}$ 
by zero) but also
very small {\em relative} errors for $x>5/2$.
\end{Remark}
\begin{Definition}{\rm 
\label{Def1}
Let $a_{l} = a_0 -\rho_0$, $a_r = a_0+\rho_0$, $b_l = b_0 - 2 \rho_0$,
$b_r = b_0 + 2 \rho_0$, $c_l = c_0 - 2 \rho_0$ and $c_r = c_0 + 2 \rho_0$. We see that $(a,b,c) \in \mathcal{S}$, implies that $a \in [a_l, a_r]$, $b \in [b_l, b_r]$,
$c \in (c_l, c_r)$.
We define $t_m = \frac{a}{2} \left ( \frac{5}{2}
+ \frac{b}{a} \right )^2 $ and note that $x \in [\frac{5}{2}, \infty )$ corresponds to
$t \in [t_m, \infty )$ and  $t_m \in 
\left ( \frac{a_l}{2} 
\left [ \frac{5}{2} + \frac{b_l}{a_l} \right ]^2 , 
\frac{a_r}{2} 
\left [ \frac{5}{2} + \frac{b_r}{a_r} \right ]^2 \right )  
=: \left (t_{m,l}, t_{m,r} \right ) = \left ( 1.998859\cdots, 1.999438 \cdots \right )$. 
}\end{Definition}

\begin{Remark}{\rm 
The error bounds proved for ${E}$ in Theorem \ref{Thm1}
are likely a 10 fold over-estimate. Comparison with the 
numerically calculated $F$ suggests that
${|}F-F_0 {|}$, ${|}F^{\prime}-F_0^\prime {|}$ and $
{|}F^{\prime \prime} - F_0^{\prime \prime} {|}$ in $\left [0, \frac{5}{2} \right ]$
are at most $2 \times 10^{-7}, 2 \times 10^{-7}$ and $5 \times 10^{-7}$ respectively. Using the nonrigorous bounds on ${E}$ and
its derivatives reduces the 
$\rho_0$ in the definition of $\mathcal{S}$ 
from $5 \times 10^{-5}$ to
$1.4 \times 10^{-5}$. 
It is thus likely that
$(a, b, c) \approx (a_0, b_0, c_0)$ with
five (rather than the proven four) digits accuracy.}
\end{Remark}

The proof of Theorem \ref{Thm1} rests on the following three propositions, proved later in the paper.
\begin{Proposition}
\label{Prop2}
The error term 
$E(x) = F(x) - F_0 (x)$  
verifies the equation
\begin{equation}
\label{eqE}
\mathcal{L} [E] := 
E^{\prime \prime} - F_0 E^{\prime \prime} - F_0^{\prime \prime} E = - F_0^{\prime \prime \prime} - F_0 F_0^{\prime \prime} 
- E E^{\prime \prime}
\end{equation}
\begin{equation}
\label{eqEBC}
E(0) = 0 = E^{\prime} (0) = E^{\prime \prime} (0) 
\end{equation}
and satisfies the bounds 
(\ref{errBoundsE}) on $I = \left [ 0, \frac{5}{2} \right ]$
\end{Proposition}
\begin{Proposition}
\label{Prop3}
For given $(a, b, c)$ with $a > 0$, $ |c| < \frac{1}{4}$, in the domain 
$x \ge - \frac{b}{a} + \sqrt{\frac{2T}{a}}$, for $T \ge 1.99$, which
corresponds to 
$t \ge T \ge 1.99$, there exists unique solution to (\ref{3}) in the form 
\begin{equation}
\label{0.0}
F(x) = a x + b + \sqrt{\frac{a}{2 t(x)}} q (t(x))
\end{equation}
that satisfies the condition $\lim_{t \rightarrow \infty} \frac{q (t)}{\sqrt{t}} \rightarrow 0$.
Furthermore, 
\begin{equation}
q(t) = q_0 (t) + \mathcal{E} (t) 
\end{equation}
where $\mathcal{E}$ is small and satisfies the following error bounds:
\begin{equation}
\Big | \mathcal{E} (t) \Big | 
\le 1.6667 \times 10^{-4}  \frac{e^{-3 t}}{9 t^{3/2}}     
\end{equation}
\begin{equation}
\Big | \mathcal{E}^{\prime} (t) - 
\frac{1}{2t} \mathcal{E} (t) \Big |
\le 1.6667 \times 10^{-4}  \frac{e^{-3 t}}{3 t^{3/2}}     
\end{equation}
\begin{equation}
\Big | \sqrt{t} \mathcal{E}^{\prime \prime} (t) - 
\frac{1}{\sqrt{t}} \mathcal{E}^\prime (t) + \frac{1}{2 t^{3/2}} \mathcal{E} (t) \Big | 
\le 1.6667 \times 10^{-4}  t^{-1} e^{-3 t}    
\end{equation}
\end{Proposition}
\begin{Proposition}
\label{Prop4}
There exists a unique triple $(a, b, c) \in \mathcal{S}$ so that the functions 
in the previous two propositions: 
$F_0 (x) + E(x)$ for $x \le \frac{5}{2}$ 
and $ax+b + \sqrt{\frac{a}{2t(x)}} q (t(x)) $ for $x \ge \frac{5}{2}$ 
and their first two derivatives
agree at $x=\frac{5}{2}$.
\end{Proposition}

\noindent{\bf The proof} of Theorem \ref{Thm1} follows from Propositions 
\ref{Prop2}-\ref{Prop4} in the following manner:
Proposition \ref{Prop2} implies $F (x) =F_0 (x) + E(x)$ satisfies 
\eqref{3}-\eqref{4} for $x \in I$; we note that $F_0 (0)=0 = F_0^\prime (0)$ 
and $F_0^{\prime\prime} (0)=1$. 

Proposition \ref{Prop3} implies $ F(x) = a x + b + \sqrt{\frac{a}{2 t(x)}} 
\left [ q_0 (t(x) ) + \mathcal{E} (t (x)) \right ]$
satisfies \eqref{3} in a range of $x$ that includes $[\frac{5}{2}, \infty)$
when $(a, b, c) \in \mathcal{S}$. Further, Proposition \ref{Prop4} ensures
that this is the same solution of the ODE \eqref{3} as the one in Proposition
\ref{Prop2}.
Identifying $F_0 (x) $ and $E(x)$ in 
Theorem \ref{Thm1} in this range of $x$ with 
$ a x + b + \sqrt{\frac{a}{2 t(x)}} q_0 (t(x))$ and  
$\sqrt{\frac{a}{2 t(x)}} \mathcal{E} (t(x))$, respectively,
and relating $x$-derivatives to $t$ derivatives,
the error bounds for $E$, $E^\prime$ and $E^{\prime \prime}$ follow from the ones
given for $\mathcal{E}$ in Proposition \ref{Prop4} for
$(a,b,c) \in \mathcal{S}$. The proofs of Propositions \ref{Prop2}-\ref{Prop4} are given
in Sections \ref{S1}-\ref{S3} respectively.

\section{Solution in the interval $ I = [0, \frac{5}{2}]$ 
and proof of Proposition \ref{Prop2}}\label{S1}

The ansatz for $F_0$ in the compact set $I$ is obtained simply by projecting an empirically obtained high accuracy approximate solution
on the subspace spanned by the first few Chebyshev polynomials. 
More precisely, to
avoid estimating derivatives of an 
approximation, which are not well-controlled,  we 
project instead the approximate 
third derivative $F^{\prime\prime\prime}=-F^{\prime\prime}F$ on 
the interval\footnote{Chebyshev polynomial
approximation is on $[-1, 1]$ interval, but a linear scaling and shift
of independent variables can accommodate any finite interval}
$I=[0, \frac{5}{2}]$. The rigorous control of the errors of the {\em integrals} 
of $F'''$ is a much simpler task, and this is the way we prove Proposition \ref{Prop2}.
For a given polynomial degree, a 
Chebyshev polynomial approximation of a function is known to be, 
typically, the 
most accurate in polynomial representation, in the sense of $L^\infty$.

We seek to control the error term $E$ in (\ref{5}) by first estimating the remainder
\begin{equation}
\label{12.0}
R(x) = F_0^{\prime \prime\prime} (x) + F_0 (x) F_0^{\prime \prime} (x),
\end{equation} 
which will be shown to be small ($ \le 0.673 \times 10^{-6}$). 
Then, we invert the principal part
of the linear part of the equation for the error term $E$ by using initial
conditions to obtain a nonlinear integral equation. 
The smallness of $R$ and careful bounds on the resolvent allow for using a contraction
mapping argument to obtain the sharp
estimates for $E$ and its derivatives stated in Proposition \ref{Prop2}.

\subsection{Estimating size of remainder $R (x)$ for $ x\in I $}\label{R(x)}
Since $P$ is a polynomial of degree twelve,
$R(x)$ is a polynomial of degree 30. 
We estimate $R$ in $I$ in the following manner.
We break
up the interval into subintervals 
$\left \{ [x_{j-1}, x_{j}] \right \}_{j=1}^{14}$
with $x_0=0$ and $x_{14} = \frac{5}{2}$, while 
$\left \{x_j \right \}_{j=1}^{13}$ is given by
$$ \left \{ 0.0625, 0.125, 0.25, 0.375, 0.50, 0.75, 1.0, x_c, 1.5, 1.75, 2.0, 2.25, 2.40 \right \}$$
where $x_c = 1.322040$.\footnote{As will be found later, it is convenient to choose one
of the subdivision points $x_c$ to be approximately, to the number of digits quoted,
the value of $x$ where 
$F_0^{\prime \prime} (x) - 2 F_0^\prime (x) +1$ changes sign.}
The intervals were chosen based 
on how rapidly the polynomial $R(x)$ varies locally.

We re-expand $R(x)$ as polynomial in the scaled variable   
$\tau$, where $x=\frac{1}{2} \left (x_j
+ x_{j-1} \right ) + \frac{1}{2} \left ( x_j-x_{j-1} \right ) \tau$.  
and write
$$ R(x) = P_3^{(j)}(\tau) + \sum_{k=4}^{30} a_k^{(j)} \tau^k
$$
and determine the maximum $M_j$ and minimum $m_j$ 
of the third degree polynomial $P_{3}^{(j)} (t)$
for $\tau \in [-1,1]$ (using simple calculus). 
We estimate the $l^1$ error
on the remaining coefficients:  
$$ E_{R}^{(j)} \equiv \sum_{k=4}^{30} \Big | a_k^{(j)} \Big | $$  
It follows that in the $j$-subinterval we have
$$ m_j - E_R^{(j)} \le R(x) 
\le M_j + E_R^{(j)} $$
The maximum and minimum over any union of subintervals
is found simply taking $\min$ and $\max$ of $m_j - E_{R}^{(j)}$ and $M_j + E_{R}^{(j)}$ 
over the the indices $j$ for 
subintervals involved.   
This elementary though tedious calculation\footnote{The maximum and minimum found
through analysis described here is found to be consistent with a 
numerical plot of the graph of $R(x)$, as must be the case. 
The calculations can be conveniently done with a computer algebra program, as they only involve operations with rational numbers.} yields
\begin{multline}
\label{12.1}
-3.22 \times 10^{-7} \le R (x) \le 2.505 
\times 10^{-7} ~{\rm for} ~x \in [0, x_c ] \\
 4.6 \times 10^{-8}  \le  R (x) \le 4.06 \times 10^{-7}  
~{\rm for} ~x \in [x_c, 2.0] \\
 2.78 
\times 10^{-7}  
\le  R(x)  \le 6.73 \times 10^{-7}  ~{\rm for} ~x \in [2.0, 2.5]  
\end{multline}  
We note that the remainder is at most
$6.73 \times 10^{-7}$ in absolute value in the interval $I$.
In a similar manner of finding maximum and minimum, 
bounds for the polynomials $F_0 (x)$, $F_0^\prime$ and $F_0^{\prime \prime} (x)$
may also be found.
For $x \in \left [ 0 , \frac{1}{8} \right ]$,
\begin{multline}
\label{12.1.0}
-5 \times 10^{-10} \le F_0 (x) \le 0.008
\ ,  -8 \times 10^{-12} \le F_0^{\prime} (x) \le  
0.13   \ , \\  
0.99 \le F_0^{\prime \prime} (x)  \le 1 + 2 \times 10^{-9} 
\end{multline}
while for $x \in \left [ \frac{1}{8}, \frac{5}{2} \right ]$,
\begin{multline}
\label{12.1.1}
0.03 \le F_0 (x) \le 2.59   \ ,     0.12  \le F_0^\prime (x)  \le 1.7  \ ,
0.09 \le  F_0^{\prime \prime} (x) \le  1
\end{multline}
\subsection{Properties of some functions used in the sequel}\label{S2.2}
Based on the calculations above, one can also conclude that  
$F_0^{\prime \prime} - 2 F_0 + 1$ and $2 F_0^{\prime \prime} - 2 F_0 $
have only one zero in the interval $I$ in the following manner. 
Note that derivatives of these 
functions are
$ - F_0 F_0^{\prime \prime} - 2 F_0^\prime + R < 0$  and 
$- 2 F_0 F_0^{\prime \prime} - 2 F_0^\prime + 2 R < 0$ respectively
in $ \left [ \frac{1}{8}, \frac{5}{2} \right ]$, where $R$ has been
estimated in the previous subsection.
From the bounds in the interval $\left [0, \frac{1}{8} \right ]$,
it is clear that  
$F_0^{\prime \prime} - 2 F_0 + 1$ and $2 F_0^{\prime \prime} - 2 F_0 $
are positive in $\left [0, \frac{1}{8} \right ]$. Thus, we conclude that
 $F_0^{\prime \prime} - 2 F_0 + 1$ and 
$2 F_0^{\prime \prime} - 2 F_0 $
can have at most one zero in $I$. The values of $F_0^{\prime \prime} - 2 F_0 
+ 1 $ at $x_c = 1.322040$
and $1.322041$ have opposite signs, implying that there is 
a unique zero in $I$ in between
two numbers (recalling that  the derivative is negative). 
Similarly, we conclude there is a unique zero of 
$2 F_0^{\prime \prime} - 2 F_0 $
 between $1.2314283$ and $1.2314284$.
Similar arguments show that $F_0^{\prime \prime} (x) - 2 F_0 (x)$ only has one zero in
$I$ at $ x = 0.9399325 \cdots < x_c $.

\subsection{Green's function estimate for $x \in [x_l, x_r] $}

Consider now the problem of solving the linear generally inhomogeneous equation
\begin{equation}
\label{12.0.0}
\mathcal{L} [ \phi ] :=\phi^{\prime\prime \prime} (x) + F_0 (x) \phi^{\prime \prime} (x) + F_0^{\prime \prime} (x)  \phi (x) = r(x)
\end{equation}
over a typical subinterval $[x_l, x_r] \subset I$,
with initial conditions $\phi (x_l)$, $\phi^\prime (x_l)$ and $\phi^{\prime \prime} (x_l)$ known.
The solution of this inhomogeneous equation is 
given by the standard variation of parameter formula:
\begin{equation}
\label{12.0.1}
\phi (x) = \sum_{j=1}^3 \phi^{(j-1)} (x_l) \Phi_j (x)  + \sum_{j=1}^3 \Phi_j (x) \int_{x_l}^x \Psi_j (t) r(t) dt 
\end{equation}
where $\left \{ \Phi_j \right \}_{j=1}^3 $ form a 
fundamental set of solutions to $\mathcal{L} \phi=0$ 
and $\left \{\Psi_j (x) \right \}_{j=1}^3$ are elements 
of the inverse of the fundamental matrix
constructed from the $\Phi_j$ and their derivatives. The 
precise expressions are unimportant in the ensuing: 
we only need their smoothness in  $x$.
It also follows from the  properties of
$\Phi_j$ and $\Psi_j$ \footnote{In particular, 
$\sum_{j=1}^3 \Phi_j (x) \Psi_j (x) =0$, 
$\sum_{j=1}^3 \Phi_j^\prime (x) \Psi_j (x)=0$}
that
\begin{equation}
\label{12.3.1}
\phi^{\prime \prime} (x) = 
\sum_{j=1}^3 \phi^{(j-1)} (x_l) \Phi_j^{\prime \prime} (x) 
+ \sum_{j=1}^3 \Phi_j^{\prime \prime} (x) \int_{x_l}^x \Psi_j (t) r(t) dt \ , 
\end{equation}
It is useful to write (\ref{12.3.1}) in the following abstract form
\begin{equation}
\label{12.4}
\phi^{\prime \prime} (x) = \sum_{j=1}^3 \phi^{(j-1)} (x_l) \Phi_j^{\prime \prime} (x) + \mathcal{G} \left [ r \right ] (x)      
\end{equation}
where from general properties of fundamental
matrix and its inverse for the linear ODEs with smooth 
(in this case polynomial) coefficients 
$\mathcal{G}$ is a bounded linear operator 
on $C([x_l, x_r])$; 
denote its norm by $M$,
\begin{equation}
  \label{eq:defM}
  M=\|\mathcal{G}\|
\end{equation}
 Then, on the interval $[x_l, x_r]$ we have,
\begin{equation}
\label{12.5}
\| \phi^{\prime \prime} \|_\infty \le  \sum_{j=1}^\infty M_j \Big | \phi^{(j-1)} (x_l) \Big | 
+ M \| r \|_{\infty};\ M_j
= \sup_{x \in [x_l,x_r]} 
\Big | \Phi_j^{\prime \prime} (x) \Big | 
\end{equation}
We will now estimate
$M_j$ for $j=1..3$ and
$M$ indirectly, 
using ``energy'' bounds. Because of 
linearity of the problem, for the purposes
of determining these bounds, it is enough to separately 
consider the cases (i)--(iii), when
$r =0$, $\phi^{(k-1)} (x_l) =0$ for $1 \le k \ne j \le 3$, 
$\phi^{(j-1)} (x_l) = 1$  respectively, and, finally, (iv) when
$\phi^{(k-1)} (x_l) =0$ for $k =1,..3$ and $r(t) \ne 0$. 

For all cases (i)-(iv), 
we return to the ODE
\begin{equation}
\label{12.13}
\phi^{\prime \prime \prime} + F_0 \phi^{\prime \prime} + F_0^{\prime \prime} \phi = r 
\end{equation}
Multiplying  by $2 \phi^{\prime \prime}$, 
integrating from $x_l$ to $x$
and using initial conditions, it follows that
\begin{equation}
\label{12.14}
\left ( \phi^{\prime \prime} (x) \right )^2 = \left ( \phi^{\prime \prime} (x_l) \right )^2 
- \int_{x_l}^x \left \{ 2 F_0 (y) \left ( \phi^{\prime \prime} (y) \right )^2  
+
2 F_0^{\prime \prime} (y) \phi^{\prime \prime} (y) \phi (y) - 
2 \phi^{\prime \prime} (y) r(y) \right \} dy,
\end{equation}     
We note further that, given  $\phi (x_l)$ and $\phi^\prime (x_l)$,
$\phi(x)$ is determined from $\phi^{\prime \prime} (x)$
and the relation
\begin{equation}
\label{12.15}
{\tilde \phi} (x) := 
\phi(x) - \phi (x_l) - (x-x_l) \phi^\prime (x_l) = \int_{x_l}^x (x-y) \phi^{\prime \prime} (y) dy
\end{equation}
Using (\ref{12.15}) in (\ref{12.14}), it follows that 
\begin{multline}
\label{12.14.1}
\left ( \phi^{\prime \prime} (x) \right )^2 = \left ( \phi^{\prime \prime} (x_l) \right )^2 
- \int_{x_l}^x 2 F_0^{\prime \prime} (y) 
\left [ \phi(x_l) + (y-x_l) \phi^{\prime} (x_l) \right ] \phi^{\prime \prime} (y) dy \\
- \int_{x_l}^x \left \{ 2 F_0 (y) \left ( \phi^{\prime \prime} (y) \right )^2  
+
2 F_0^{\prime \prime} (y) \phi^{\prime \prime} (y) {\tilde \phi} (y) - 
2 \phi^{\prime \prime} (y) r(y) \right \} dy,
\end{multline}     

\subsection{Case (i): Determination of $M_1$}

In this case, we set 
$\phi(x_l)=1$, $\phi^{\prime} (x_l) = \phi^{\prime \prime} (x_l) = 0 = r(x)$ 
in (\ref{12.14.1}) to obtain (using $F_0^{\prime \prime} > 0$ 
and Cauchy-Schwartz):
\begin{multline}
\label{12.14.2}
\left ( \phi^{\prime \prime} (x) \right )^2 \le  
\int_{x_l}^x F_0^{\prime \prime} (y) dy +
\int_{x_l}^x F_0^{\prime \prime} (y) {\tilde \phi}^2 (y) dy \\
+\int_{x_l}^x \left ( \phi^{\prime \prime} (y) \right )^2 \left \{ 2 F_0^{\prime \prime} (y) 
- 2 F_0 (y) \right \} dy  
\end{multline}
It is convenient to define, for $ x_l \le \tau \le x_r $,
\begin{equation}
\label{12.17}
D(\tau) = \sup_{x \in [x_l, \tau]} \Big | \phi^{\prime \prime} \Big |^2 (x) 
\end{equation}
From (\ref{12.15}) and the definition of $D$ we get 
\begin{equation}
\label{12.17.0}
\| \phi^{\prime \prime} \|_\infty \le \sqrt{D (x_r)} \ ,
\Big | {\tilde \phi} (x) \Big |  
\le \frac{(x-x_l)^2}{2} \sqrt{D(x)}  \ , ~~~\left ( \phi^{\prime \prime} (x) \right )^2 \le D (x);\  \text{in }[x_l, x_r], 
\end{equation}
Using \eqref{12.17.0} and (\ref{12.14.2}) we see that
\begin{multline}
\label{12.14.3}
\left ( \phi^{\prime \prime} (x) \right )^2 \le  
\left ( \int_{x_l}^x F_0^{\prime \prime} (y) dy \right ) +
\int_{x_l}^x F_0^{\prime \prime} (y) \frac{(y-x_l)^4}{4} D (y) dy \\
+\int_{x_l}^x \left ( \phi^{\prime \prime} (y) \right )^2 \left \{ 2 F_0^{\prime \prime} (y) 
- 2 F_0 (y) \right \} dy  
\end{multline}
Define now
\begin{multline}
\label{12.14.6.0}
Q_1 (x) = F_0^{\prime \prime} (x) \left ( 2 + \frac{(x-x_l)^4}{4} \right ) - 2 F_0 (x)   
~~{\rm if}~~2 F_0^{\prime \prime} (x) - 2 F_0 (x) > 0   \\
\text{and} 
\ Q_1 (x) = \frac{(x-x_l)^4}{4} F_0^{\prime \prime} (x) ~~{\rm if}~~2 F_0^{\prime \prime} (x) - 2 F_0 (x) \le 0   
\end{multline} 
Equations
(\ref{12.17.0}) 
and (\ref{12.14.3}) imply the following inequality for
for $0 \le x_l \le x \le \tau \le x_r$
\begin{equation}
\label{12.14.6}
\left ( \phi^{\prime \prime} (x) \right )^2  
\le \int_{x_l}^{\tau} F_0^{\prime \prime} (y) dy 
+ \int_{x_l}^\tau D(y) Q_1(y) dy 
\end{equation} 
Since the right side is independent of $x$, it follows that 
\begin{equation}
\label{12.14.7} 
D(\tau) \le 
\int_{x_l}^{\tau} F_0^{\prime \prime} (y) dy 
+\int_{x_l}^\tau D(y) Q_1 (y) dy 
\end{equation}
From Gronwall's lemma, it follows that
\begin{equation}
\label{12.4.8}
D(\tau) \le 
\left ( \int_{x_l}^{\tau} F_0^{\prime \prime} (y) dy \right ) \exp \left [ \int_{x_l}^\tau Q_1 (y) dy \right ] 
\end{equation}
Evaluating \eqref{12.4.8} at $\tau=x_r$ immediately implies
\begin{equation}
\label{12.14.9}
M_1 \le 
\left ( F_0^\prime (x_r) - F_0^{\prime} (x_l) \right )^{1/2}
\exp \left [ \frac{1}{2} \int_{x_l}^{x_r} Q_1 (y) dy \right ] 
\end{equation}

\subsection{Case (ii): Determination of $M_2$}

In this case, we set 
$\phi^\prime (x_l)=1$, $\phi (x_l) = 0 = \phi^{\prime \prime} (x_l) = 0 = r(x)$ 
in (\ref{12.14.1}) to obtain by Cauchy-Schwartz
\begin{multline}
\label{12.14.10.0}
\left ( \phi^{\prime \prime} (x) \right )^2  
\le \left ( \int_{x_l}^x F_0^{\prime \prime} (y) (y-x_l)^2 dy \right ) 
+\int_{x_l}^x F_0^{\prime \prime} (y) {\tilde \phi}^2 (y) dy \\
+\int_{x_l}^x \left ( \phi^{\prime \prime} (y) \right )^2 \left \{ 2 F_0^{\prime \prime} (y) 
- 2 F_0 (y) \right \} dy  
\end{multline}
Again introducing $D(x) $ and $Q_1 (x)$ as in case (i), we obtain the
inequality
\begin{equation}
\label{12.14.11} 
D(\tau) \le 
\int_{x_l}^{\tau} (y-x_l)^2 F_0^{\prime \prime} (y) dy 
+\int_{x_l}^\tau D(y) Q_1 (y) dy 
\end{equation}
and therefore, it follows from Gronwall's Lemma that 
\begin{equation}
\label{12.14.12.0}
M_2 \le 
\left ( \int_{x_l}^{x_r} (y-x_l)^2 F_0^{\prime \prime} (y) dy\right )^{1/2}
\exp \left [ \frac{1}{2} \int_{x_l}^{x_r} Q_1 (y) dy \right ] 
\end{equation}

\subsection{Case (iii): Determination of $M_3$}
In this case, we set 
$\phi^{\prime \prime} (x_l)=1$, $\phi (x_l) = 0 = \phi^{\prime} (x_l) = 0 = r(x)$ 
in (\ref{12.14.1}) to obtain 
\begin{equation}
\label{12.14.10}
\left ( \phi^{\prime \prime} (x) \right )^2  
\le 1 
+\int_{x_l}^x F_0^{\prime \prime} (y) {\tilde \phi}^2 (y) dy 
+\int_{x_l}^x \left ( \phi^{\prime \prime} (y) \right )^2 \left \{ 2 F_0^{\prime \prime} (y) 
- 2 F_0 (y) \right \} dy  
\end{equation}
Once again introducing $D$ as in case (i) and defining
\begin{multline}
\label{12.14.0}
Q_2 (x) = \left ( 1 + \frac{(x-x_l)^4}{4} \right ) F_0^{\prime \prime} (x) - 2 F_0 (x)  \ , 
~~{\rm if} ~~F_0^{\prime \prime} (x) - 2 F_0 (x) > 0  \\
\text{and} \ Q_2 (x) = \frac{(x-x_l)^4}{4} F_0^{\prime \prime} (x) \ , 
~~{\rm if} ~~F_0^{\prime \prime} (x) - 2 F_0 (x) \le 0  \ ,
\end{multline}
(\ref{12.14.10}) imply
\begin{equation}
\label{12.14.12} 
D(\tau) \le 1  
+\int_{x_l}^\tau D(y) Q_2 (y) dy 
\end{equation}
Gronwall's Lemma and definition of $D$ implies
\begin{equation}
\label{12.14.14}
M_3 \le 
\exp \left [ \frac{1}{2} \int_{x_l}^{x_r} Q_2 (y) dy \right ] 
\end{equation}

\subsection{Case (iv): Determination of $M = \| \mathcal{G} \|$}
With 
$\phi (x_l) = 0 = \phi^{\prime} (x_l) =\phi^{\prime \prime} (x_l) = 0 $, 
(\ref{12.14.1}) implies by Cauchy-Schwartz
\begin{multline}
\label{12.14.15} 
\left ( \phi^{\prime \prime} (x) \right )^2  
\le  
\int_{x_l}^x r^2 (y) dy 
+\int_{x_l}^x F_0^{\prime \prime} (y) {\tilde \phi}^2 (y) dy \\
+\int_{x_l}^x \left ( \phi^{\prime \prime} (y) \right )^2 \left \{ F_0^{\prime \prime} (y) 
- 2 F_0 (y) + 1 \right \} dy  
\end{multline}
We define $D$  as in (i) and also 
\begin{multline}
\label{12.14.10.0.0}
Q(x) = F_0^{\prime \prime} (x) - 2 F_0 (x) + 1 + \frac{(x-x_l)^4}{4} F_0^{\prime \prime} (x) ~~{\rm if} ~F_0^{\prime \prime} - 2 F_0 +1 \ge 0 
\\
Q(x) = \frac{(x-x_l)^4}{4} F_0^{\prime \prime} (x) ~~{\rm if} ~F_0^{\prime \prime} - 2 F_0 +1 < 0 
\end{multline}
The inequality 
\begin{equation}
\label{12.14.16} 
D(\tau) \le \int_{x_l}^\tau r^2 (y) dy 
+\int_{x_l}^\tau D(y) Q (y) dy 
\end{equation}
follows from 
\eqref{12.14.15},  
implying from Gronwall's Lemma
\begin{equation}
\label{12.14.17}
\sqrt{D(x_r)} \le \| r \|_\infty  
\left ( x_r - x_l \right )^{1/2}
\exp \left [ \frac{1}{2} \int_{x_l}^{x_r} Q(y) dy \right ] 
\end{equation}
implying
\begin{equation}
\label{12.14.18}
M   
\le \left ( x_r - x_l \right )^{1/2}
\exp \left [ \frac{1}{2} \int_{x_l}^{x_r} Q(y) dy \right ] 
\end{equation}

\subsection{Existence of $F$; error bounds for $x \in [x_l, x_r] \subset I$}
Consider the decomposition
\begin{equation}
\label{12.1.0.0}
F(x) = F_0 (x) + E(x) 
\end{equation}
We seek to find error estimates for $E(x)$ and its first two derivatives for
$x \in I$. For this purpose we break up $I$ into
a number of subintervals. Note that for the first
subinterval $x_l=0$, where
$E(x_l)=0=E^\prime (x_l) = E^{\prime \prime} (x_l) $. 
Consider 
a typical subinterval $I = [x_l, x_r] $ where the bounds on 
$E(x_l)$, $E^\prime (x_l)$ and $E^{\prime \prime} (x_l)$ 
on earlier subintervals have been already obtained. 
The equation for $E(x)$ on $[x_l, x_r]$ is
\begin{equation}
\label{12.2}
\mathcal{L} [ E ] :=E^{\prime\prime \prime} (x) + F_0 (x) E^{\prime \prime} (x) + F_0^{\prime \prime} (x) 
E(x) = 
-E (x) E^{\prime \prime} (x) - R(x) 
\end{equation}
Inverting $\mathcal{L}$ as described in previous subsection leads to
the following integral equation:
\begin{equation}
\label{12.3}
E^{\prime \prime} (x) = \sum_{j=1}^3 E^{(j-1)} (x_l) \Phi_j^{\prime \prime} (x)   
-\mathcal{G} \left [ R \right ] (x) +
\mathcal{G} \left [ E E^{\prime \prime} \right ] (x) 
=: \mathcal{N} \left [ E^{\prime \prime} \right ] (x)     
\end{equation}
and where  $E$ is given by
\begin{equation}
\label{12.8.0}
E (x) - E(x_l) - (x-x_l) E^\prime (x_l) =: {\tilde E} (x) =  \int_{x_l}^x (x-t) E^{\prime \prime} (t) dt 
\end{equation}
Note that \eqref{12.8.0} implies
\begin{equation}
\label{12.9}
\| {\tilde E} \|_\infty \le \frac{(x_r-x_l)^2}{2} \| E^{\prime \prime} \|_\infty 
\end{equation}
We prove the following Lemma that, 
once some bounds are satisfied, ensures the existence,  
uniqueness and smoothness of the solution 
$E$ of \eqref{12.2} and 
provides estimates of $E$, $E^\prime$ and $E^{\prime\prime}$.
\begin{Lemma}\label{Lemma5}{\rm Assume that for some  $\epsilon>0$ we have
\begin{equation}
\label{12.9.2}
M \left ( \Big | E (x_l) \Big | + (x_r-x_l) \Big | E^\prime (x_l) \Big | \right ) (1+\epsilon)
+ \frac{1}{2} (x_r-x_l)^2 M B_0 (1+\epsilon)^2 < \epsilon   \ ,
\end{equation}
\begin{equation}
\label{12.9.3}
M \left ( \Big | E (x_l) \Big | + (x_r-x_l) \Big | E^\prime (x_l) \Big | \right ) 
+ (x_r-x_l)^2 M B_0 (1+\epsilon) < 1 \ , 
\end{equation}
where
\begin{equation}
\label{12.9.0.0}
B_0 = M \| R \|_\infty + \sum_{j=1}^3 M_j \Big | E^{(j)} (x_l) \Big | .
\end{equation}
\label{lem1}
Then, there exists a unique solution $E^{\prime \prime} $ of \eqref{12.3} 
in a ball of radius $ B_0 (1+\epsilon) $ in the sup norm in $C([x_l, x_r])$.

Under these assumptions, $E$ is in $C^3([x_l,x_r])$} and satisfies 
(\ref{12.2}) with initial conditions $E^{(j)} (x_l)$, $j=0,1,2$
and
\begin{multline}
\label{12.12}
\| E^{\prime \prime \prime} \|_\infty \le 
\| F_0 \|_\infty (1+\epsilon) B_0
+ \| F_0^{\prime \prime} \|_\infty \left ( \Big | E (x_l) \Big |
+ (x_r - x_l ) \Big | E^\prime (x_l) \Big | \right ) \\ 
+ \frac{1}{2} \| F_0^{\prime \prime} \|_\infty (x_r-x_l)^2 B_0 (1+\epsilon)  
+ \frac{1}{2} (x_r-x_l)^2 B_0^2 (1+\epsilon)^2 +
\| R \|_\infty  
\end{multline}
\end{Lemma}
\begin{proof}
Since $\mathcal{G}$ is the Green's function of a smooth linear ODE, $\mathcal{G}$ maps $C([x_l,x_r])$ into itself; 
the same, clearly, holds for $\mathcal{N}$.
From the definitions of $M$ in \eqref{eq:defM}, and of $M_j$,  $j=1,2,3$  in  \eqref{12.5} (whose bounds will be obtained using
\eqref{12.14.9}, \eqref{12.14.12.0}, \eqref{12.14.14} and \eqref{12.14.18})  it follows that
\begin{multline}
\label{12.8}
\| \mathcal{N} \left [ E^{\prime \prime} \right ] 
\|_\infty
\le M \| R \|_\infty + 
\sum_{j=1}^3 M_j \Big | E^{(j-1)}(x_l) \Big | \\ 
+ M \left ( | E (x_l) | 
+ |x_r-x_l| | E^\prime (x_l) | \right ) 
\| E^{\prime \prime} \|_\infty
+\frac{(x_r-x_l)^2}{2} M \| E^{\prime \prime} \|_\infty^2    
\end{multline}
and 
\begin{multline}\label{12.8.1}
\| \mathcal{N} \left [ E^{\prime \prime} \right ] 
- \mathcal{N} \left [ {\hat E}^{\prime \prime} \right ] 
\|_\infty \le  
M \left ( | E (x_l) | 
+ |x_r-x_l| | E^\prime (x_l) | \right ) 
\| E^{\prime \prime} - {\hat E}^{\prime \prime} \|_\infty \\
+\frac{(x_r-x_l)^2}{2} 
M \left ( \| E^{\prime \prime} \| + \| {\hat E}^{\prime \prime} \|_\infty    
\right ) 
\| E^{\prime \prime} -{\hat E}^{\prime \prime} \|_\infty    
\end{multline}
Using \eqref{12.9}, \eqref{12.9.2} and \eqref{12.9.0.0} in \eqref{12.8}  and \eqref{12.8.1} we see that $\mathcal{N}$ is contractive in a ball of radius $(1+\epsilon) B_0$ in $C([x_l,x_r])$, implying existence and uniqueness of a solution to \eqref{12.8.0}.  Clearly,  (\ref{12.3}) is equivalent
to (\ref{12.2}); from (\ref{12.2}) it
follows that $E^{\prime \prime \prime}$ is also continuous. Now,   $E^{\prime \prime \prime}$ is easily estimated from (\ref{12.2}) in terms
of lower order derivatives, and the result follows. 
\end{proof}

\subsection{Determining $E$ using Lemma \ref{Lemma5}}
In this section, we break the interval $[0,5/2]$ in a suitable way and show that Lemma \ref{Lemma5} applies in all subintervals. 
The choice of subintervals is $\mathcal{I}_1 = [0, x_c ]$
$\mathcal{I}_2 = [x_c, 2]$, 
$\mathcal{I}_3 = \left [2 , \frac{5}{2} \right ]$, where $x_c = 1.322040$ is, within number of digits quoted, 
the zero of $F_0^{\prime \prime} - 2 F_0 +1 $ obtained in \S\ref{S2.2}.

\subsubsection{Error estimates on $\mathcal{I}_1$}
On $\mathcal{I}_1 $, using  \eqref{6.0}, it is easily checked that 
\begin{equation}
 M \le 3.03
\end{equation}
while by \eqref{12.1} we have $\| R \|_{\infty, \mathcal{I}_1} 
\le 3.22 \times 10^{-7}$
Since the initial conditions on this interval are  $E(x_l)=E^\prime (x_l)=E^{\prime \prime} (x_l)=0$, the $M_j$ do not contribute to \eqref{12.8}, \eqref{12.8.1},  and \eqref{12.9.0.0} implies
\begin{equation}
B_0 \le 0.9757 \times 10^{-6}  
\end{equation}
The conditions (\ref{12.9.2}) and (\ref{12.9.3})  
are
satisfied for $\epsilon = 3 \times 10^{-6}$ so that Lemma \ref{lem1} implies
\begin{equation}
\label{I1E2b}
\| E^{\prime \prime} \|_{\infty, \mathcal{I}_1} \le 0.976 \times 10^{-6} 
\end{equation}
On integration it follows that
\begin{equation}
\label{I1Eb}
\| E^\prime \|_{\infty, \mathcal{I}_1} \le 1.29 \times 10^{-6}  \ ,
\| E \|_{\infty, \mathcal{I}_1 } \le 0.853 \times 10^{-6} 
\end{equation}

\subsubsection{Error estimates on $\mathcal{I}_2$}

On $\mathcal{I}_2 = [x_c, 2] $, 
using \eqref{6.0}, it is easily checked that 
\begin{equation}
 M_1 \le 0.572 \ , M_2 = 0.199 \ , M_3 \le 1.01, M \le 0.825  . 
\end{equation}
Since at $x_l$, we can apply (\ref{I1E2b})
and (\ref{I1Eb}) to bound
$E(x_l), E^\prime (x_l), E^{\prime \prime} (x_l)$,
using $\| R \|_{\infty, \mathcal{I}_2} \le 4.06 \times 10^{-7}$, 
\eqref{12.9.0.0} implies 
\begin{equation}
B_0 \le 2.0653 \times 10^{-6}
\end{equation}
Lemma \ref{lem1} applies if
$\epsilon = 2\times 10^{-6}$.
Therefore, the solution $E$ exists and is unique on $\mathcal{I}_2$ and
\begin{equation}
\label{E2I2b}
\| E^{\prime \prime} \|_{\infty, \mathcal{I}_2}\le 2.07 \times 10^{-6}
\end{equation}
By integration and using the bounds on $E(x_l),E'(x_l)$ obtained in the previous interval, see  \eqref{I1Eb},  we get
\begin{equation}
\label{EI2b}
\| E^\prime \|_{\infty, \mathcal{I}_2}  \le 
2.7 \times 10^{-6}  \ ,
\| E \|_{\infty, \mathcal{I}_2 } \le 2.21 \times 10^{-6}
\end{equation}

\subsubsection{Error estimates on $\mathcal{I}_3 = \left [ 2,  2.5 \right ]$}

On $\mathcal{I}_3 = [ 2, 2.5 ]$, using \eqref{6.0} we get 
\begin{equation}
 M_1 \le 0.3 \ , M_2 \le  0.0744  \ , M_3 \le 1.01 \ , M \le 0.708  
\end{equation}
while $\| R \|_{\infty, [2,  2.5]} \le 0.673 \times 10^{-6}$   by \eqref{12.1}.
Proceeding as in the previous interval we get
\begin{equation}
B_0 \le 3.431 \times 10^{-6}
\end{equation}
Here Lemma \ref{lem1} applies with
$\epsilon = 3 \times 10^{-6}$ and thus
\begin{equation}
\label{E2I3b}
\| E^{\prime \prime} \|_{\infty, \mathcal{I}_3} \le 3.44 \times 10^{-6}   
\end{equation}
Proceeding as in the previous intervals, we get
\begin{equation}
\label{EI3b}
\| E^\prime \|_{\infty, \mathcal{I}_3} \le 
4.42 \times 10^{-6} \ , 
\| E \|_{\infty, \mathcal{I}_3}  
\le 3.99 \times 10^{-6}
\end{equation}

\subsection{End of proof of Proposition \ref{Prop2}} 
In the previous subsection we have shown that Lemma \ref{lem1} applies on 
each of the intervals $\mathcal{I}_j$ entailing 
the existence and uniqueness of a solution
$E$ of the initial value problem
(\ref{eqE})-(\ref{eqEBC}) 
over the interval $I $. The same calculations show the $L^\infty$
norms of $E^{(j)}$ satisfy the bounds in Theorem \ref{Thm1}.

\section{Solution in $t \ge T \ge 1.99 $ for $|c| < \frac{1}{4} $, $a > 0$
and proof of Proposition \ref{Prop3}}
\label{S2}

We decompose 
\begin{equation}
\label{N.1}
F(x) = a x+ b + g (x) 
\end{equation}
Then, it is clear from (\ref{3}) that $g$ satisfies
\begin{equation}
\label{N.2}
g^{\prime \prime \prime} + \left ( a x + b +g \right ) g^{\prime \prime} =0
\end{equation}
\begin{Lemma}
\label{LemN1}
Assume $a > 0$.
Then any solution $g$ to (\ref{N.2}) for which $g \rightarrow 0$ 
as $x \rightarrow
+\infty$ has the following asymptotic behavior
\begin{equation}
\label{N.5}
g^{\prime \prime} (x) \sim C \exp \left [ - \frac{a}{2} 
\left ( x + \frac{b}{a} \right )^2 \right ]
\end{equation}
\begin{equation}
\label{N.6}
g^\prime (x) \sim \frac{C}{(a x+ b)} 
\exp \left [ - \frac{a}{2} 
\left ( x + \frac{b}{a} \right )^2 \right ]
\end{equation}
\begin{equation}
\label{N.7}
g (x) \sim \frac{C}{(ax+b)^2} 
\exp \left [ - \frac{a}{2} 
\left ( x + \frac{b}{a} \right )^2 \right ]
\end{equation}
\end{Lemma}   
\begin{proof}
Eq. (\ref{N.2}) implies
\begin{equation}
\label{N.3}
g^{\prime \prime} (x) = {\tilde C} 
\exp \left [ - \frac{a}{2} \left ( x+ \frac{b}{a} \right )^2
+ \int_{x_0}^x g(t) dt \right ]
\end{equation}
Since $g(x) = o(1)$, $\int_{x_0}^x g(t) dt = o (x)$ as $x \rightarrow \infty$;
$a > 0$ and (\ref{N.3}) imply that for sufficiently large $x$
we have
\begin{equation}
\label{N.3.0}
\Big | g^{\prime \prime} (x) \Big | \le \Big | {\tilde C} \Big |
\exp \left [ - \frac{a}{2} \left ( x+ \frac{b}{a} \right )^2 + \epsilon x \right ]
\end{equation}
Integrating $g^{\prime \prime}$ and using \eqref{N.3.0}, we get for large $x$,
\begin{equation}
\label{N.3.1}
\Big | g^{\prime } (x) -C_1 \Big | \le \frac{\Big |{\tilde C} \Big |}{
ax+b -\epsilon}  
\exp \left [ - \frac{a}{2} \left ( x+ \frac{b}{a} \right )^2 + \epsilon x \right ]
\end{equation}
for some $C_1$. Similarly, for large $x$ we get
\begin{equation}
\label{N.3.2}
\Big | g (x) -C_1 x -C_2 \Big | \le \frac{\Big | {\tilde C} \Big |}{
(ax+b -\epsilon)^2}  
\exp \left [ - \frac{a}{2} \left ( x+ \frac{b}{a} \right )^2 + \epsilon x \right ]
\end{equation}
Since $g = o(1) $, we must have $C_1=0, C_2=0$, giving
rise to 
an exponentially decaying {\it a priori} bounds on 
$g$ in \eqref{N.3.2} (with $C_1=0=C_2$). We can then set
$x_0=\infty$ in (\ref{N.3}), and (\ref{N.5}) follows, and by
integration,  (\ref{N.6}) and (\ref{N.7}) hold.
\end{proof}
\begin{Lemma}
\label{LemN2}
For given $a > 0$, $b$ and $C$, for $x_0 >0$ sufficiently large, 
there exists 
unique solution
to \eqref{N.2} in $[x_0, \infty)$ with 
asymptotic behavior (\ref{N.5})-(\ref{N.7}). 
\end{Lemma}
\begin{proof}
Take $x_0 > 0$ sufficiently large so that 
\begin{equation}
\label{N.3.3.0}
\frac{e^{-a/2 (x_0 + b/a)^2}}{(ax_0+b)^3}  
< \frac{1}{4 |C|} 
\end{equation}
We define $H(x) = \exp \left [ \frac{a}{2} (x+\frac{b}{a} )^2 \right ] 
g^{\prime \prime} (x)$.
Then, from Lemma \ref{LemN1}, we see that
the appropriate space for $H$ is 
$\mathbf{C} ([x_0, \infty))$
with the sup norm.
Lemma \ref{LemN1} also implies
\begin{equation}
\label{N.3.3}
g(x) = \int_{\infty}^x \int_{\infty}^y 
\exp \left [ -\frac{a}{2} (s+\frac{b}{a} )^2 \right ] H (s) ds dy 
\end{equation} 
Eq. (\ref{N.3.3}) immediately implies 
\begin{equation}
\label{N.3.4}
|g(x) | \le  \frac{e^{-\frac{a}{2} (x+b/a)^2 }}{(ax+b)^2} \| H \|_\infty 
\end{equation}
From (\ref{N.2}) 
we obtain the following integral equation:
\begin{equation}
\label{N.3.5}
H (x) = C - \int_{\infty}^x \int_\infty^y H(s) g(s) ds dy =:\mathcal{N} [H] (x) 
\ ,
\end{equation}
where $g$ is determined from $H$ using (\ref{N.3.3}). 
We analyze (\ref{N.3.5}) in $\mathbf{C} \left ( [x_0 , \infty ) \right )$.
We consider the space of continuous functions in $[x_0, \infty)$. 
equipped with $\| . \|_\infty$ norm. Clearly 
$\mathcal{N}$ maps a ball of radius $2 |C|$ in this space
back to itself since (\ref{N.3.4}),
(\ref{N.3.5}) and (\ref{N.3.3.0}) implies
\begin{equation}
\label{N.3.6}
\| \mathcal{N} [H] \|_\infty \le |C| + 
\frac{e^{-a/2 (x_0 + b/a)^2}}{(ax_0+b)^3}  
\| H \|^2_\infty 
\le 2 |C |
\end{equation}
Thus $\mathcal{N}$ maps the ball $B_C$ of radius $2 |C|$ into itself and is contractive
there since
\begin{multline}
\label{N.3.7}
\| \mathcal{N} [ H_1 ] - \mathcal{N} [H_2] \|_\infty \le 
\frac{e^{-a/2 (x_0 + b/a)^2}}{(ax_0+b)^3}  
\left ( \| H_1\|_\infty+\|H_2 \|_\infty \right ) 
\| H_1 - H_2 \|_\infty \\
\le \frac{4 |C| e^{-a/2 (x_0 + b/a)^2}}{(ax_0+b)^3}  
\| H_1 - H_2 \|_\infty 
=:\alpha \|H_1 - H_2 \|_\infty \ ,
\end{multline}
where $\alpha < 1$.
Thus, (\ref{N.3.5}) has a unique solution in $B_C$.
Furthermore from 
\eqref{N.3.5}, 
it is clear that  $\Big | H (x) - C \Big |$ as $x\to +\infty$. 
Recalling the definition of $H$ we see that
the
asymptotic behavior of $g$ is as given by
Lemma \ref{LemN1}.
\end{proof}
\begin{Remark}{\rm 
\label{remHunique}
From Lemmas \ref{LemN1} and \ref{LemN2}, it follows that for given $a$, $b$ with
$a >0$, there exists a one parameter ($C$) family of solutions $g$ to (\ref{N.2})  
for which $g \rightarrow 0$ as $x \rightarrow \infty$. 
Any such solution has the asymptotic behavior
given 
in Lemma \ref{LemN1}}
\end{Remark}

We seek to prove Proposition \ref{Prop3}. For that purpose, 
for $a > 0$, recalling the change of variable
$t = t(x)$ in (\ref{10.2}), we make the transformation:
\begin{equation}
\label{14.1}
F(x) = a x + b +
\sqrt{\frac{a}{2t}} q(t) 
\end{equation}   
Note that the change of variable involves the parameters $a$ and $b$; 
$c$ only appears in the the solution $q(t)$ as shall be seen shortly.
The domain $t \ge T$ corresponds to $ x \ge - \frac{b}{a} + \sqrt{\frac{2 T}{a}} $.
The change of variables (\ref{14.1}) in (\ref{3}) results in $q(t)$ satisfying 
\begin{equation}
\label{14.2}
{\frac {d^{3}}{d{t}^{3}}}q + \left ( 1 +  
\frac {q}{2t} \right )  {\frac {d^{2}}{d{t}^{2}}}q
+ \left( -\frac {1}{2t}+\frac{3}{4 t^2} -
\frac {q}{4 t^{2}} \right) \frac {dq}{dt} 
+ \left ( \frac{1}{2 t^2} -\frac{3}{4 t^3} \right ) q  
 +\frac{q^2}{4 t^3} = 0 
\end{equation}
Equation (\ref{14.2}) admits two growing solutions $\sqrt{t}$ and $t$ 
corresponding to the freedom of changing $a$ and $b$. The only solution for which
$\frac{q}{\sqrt{t}} \rightarrow 0 $ 
$t \rightarrow \infty$, as noted in Lemma
\ref{LemN1}, corresponds to 
$q(t) \sim c t^{-1/2} e^{-t}$ for some $c$.

From the general theory of representation of solutions by transseries
\cite{Duke} \footnote{A slight
modification is needed to accommodate the
present ODE which violates 
a non-degeneracy condition on the eigenvalues of the linearization; the changes
are minor. Also, transseries are used to generate the appropriate ansatz and
motivate our choice of $q_0$, but play no direct role in the proofs.}
it follows that any decaying solution to (\ref{14.2})
has the following convergent function series representation 
for sufficiently large $x$:
\begin{equation}
\label{14.4}
q(t) = \sum_{n=1}^\infty \xi^n Q_n (t)  \ , {\rm where} ~\xi = \frac{c e^{-t}}{\sqrt{t}} 
\end{equation}
where the functions $Q_n$ are bounded \footnote{More precisely, $Q_n$ are the Borel sums of their asymptotic power series.}. The equations for $Q_n$ are obtained by formally plugging in (\ref{14.4}) into
(\ref{14.2}), equating the different powers of $\xi$ and requiring that $Q_n$ be
bounded as $t \rightarrow \infty$.
Only the equation for $Q_1$ is homogeneous while for $n > 1$, the equation
for $Q_n$ involves $Q_j$ for $1 \le j < n$ as a forcing term. The associated
homogeneous equation does not admit any bounded
solution, and thus the $Q_n$s are uniquely determined from their equations and the boundedness
condition.  
The multiplicative freedom of $Q_1$ is equivalent to choice of $c$ and therefore without
loss of generality, 
we may assume $Q_1 \rightarrow 1$ as $t \rightarrow \infty$. 

This motivates the choice of the approximation $q_0$ as a truncation of (\ref{14.4}) (we choose to retain two terms in the expansion). 
To prove that this approximation is accurate, we {\em define} $\mathcal{E}=q-q_0$ and show that $\mathcal{E}$ is small for $t \ge T$ in an exponentially  weighted $L^\infty$ norm. 
We thus define
\begin{equation}
\label{14.3.1} 
\mathcal{E} (t) =q(t) - q_0 (t)
\end{equation}
where 
\begin{equation}
\label{14.3.0}
q_0 (t) = \frac{c e^{-t}}{\sqrt{t}} Q_1 (t) + \frac{c^2 e^{-2t}}{t} Q_2 (t) 
\end{equation}
where 
\begin{equation}
\label{14.5}
Q_1 (t)  = 2 t I_0  (t) \ , ~{\rm where} ~I_0 (t) := 1 - \sqrt{\pi t} e^{t} {\rm erfc} (\sqrt{t}) = \frac{1}{2} \int_0^\infty \frac{e^{-st}}{(1+s)^{3/2}} ds 
\end{equation}
\begin{equation} 
\label{14.5.0}
Q_2 (t) =  
- t I_0 - t I_0^2 + 2 t J_0 \ , ~{\rm where} ~ J_0 (t) 
:= 1 - \sqrt{2 \pi t} e^{2 t} {\rm erfc} (\sqrt{2 t}) = \frac{1}{4} \int_0^\infty \frac{e^{-st}}{(1+s/2)^{3/2}} ds 
\end{equation}
\begin{Remark}
\label{remq}
\rm{ 
It is clear from (\ref{14.3.0})-\ref{14.5.0}) that
$q_0 (t) \sim \frac{c e^{-t}}{\sqrt{t}}$ as $t \rightarrow \infty$; furthermore
since, by the change of variables we have
$ g(x) = \sqrt{\frac{a}{2 t(x)}} q(t(x))$, by 
Lemma \ref{LemN1}, with $C=\sqrt{2} a^{3/2} c$,
we have
\begin{equation}
\label{eqasymptq0}
q(t) \sim q_0 (t)  ~\ , ~ q^\prime (t) \sim q_0^\prime (t) \ , 
q^{\prime \prime} (t) \sim q_0^{\prime \prime} (t) 
\ , {\rm as} ~~t \rightarrow \infty
\end{equation}
}
\end{Remark}
To analyze the fully nonlinear equation (\ref{14.2}) we write the differential equation for $\mathcal{E}$ which follows from \eqref{14.3.1} and (\ref{14.2})
\begin{multline}
\label{14.4.0}
{\frac {d^{3}}{d{t}^{3}}} \mathcal{E} + \left ( 1 +  
\frac {q_0}{2t} \right )  {\frac {d^{2}}{d{t}^{2}}} \mathcal{E}
+ \left( -\frac {1}{2t}+\frac{3}{4 t^2} -
\frac {q_0}{4 t^{2}} \right) \frac {d\mathcal{E}}{dt}  \\
+ \left ( \frac{1}{2 t^2} -\frac{3}{4 t^3} + 
\frac{q_0^{\prime \prime}}{2t}- \frac{q_0^\prime}{4 t^2}
+ \frac{q_0 (t)}{2 t^3} \right ) \mathcal{E} =   
- \frac{\mathcal{E}}{2t} {\mathcal{E}}^{\prime \prime} + \frac{\mathcal{E}}{4 t^2} {\mathcal{E}}^\prime 
-\frac{{\mathcal{E}}^2}{4 t^3}  - R 
\end{multline}
where the remainder $R=R (t)$ is given by  
\begin{multline}
\label{14.5R}
R = 
{\frac {d^{3}}{d{t}^{3}}}q_0 + \left ( 1 +  
\frac {q_0}{2t} \right )  {\frac {d^{2}}{d{t}^{2}}}q_0
+ \left( -\frac {1}{2t}+\frac{3}{4 t^2} -
\frac {q_0}{4 t^{2}} \right) \frac {dq_0}{dt} 
+ \left ( \frac{1}{2 t^2} -\frac{3}{4 t^3} \right ) q_0  
 +\frac{q_0^2}{4 t^3} \\
= \xi^3 R_3 (t) + \xi^4 R_4 (t) \ , ~{\rm where}~ \xi = c t^{-1/2} e^{-t} 
\end{multline}
where 
\begin{multline}
\label{14.5.1}
R_3 (t) =  \left ( -\frac{ 3 Q_1^\prime}{4 t^2} -\frac{Q_1^\prime}{t}  
+\frac{5 Q_1}{2 t} + \frac{Q_1^{\prime \prime}}{2 t} +  
\frac{13 Q_1}{4 t^2} + \frac{9 Q_1}{4 t^3} \right ) Q_2 - \frac{2}{t} Q_1 Q_2^\prime
+ \frac{1}{2t} Q_1 Q_2^{\prime \prime} - \frac{5}{4 t^2} Q_1 Q_2^\prime
\end{multline}
\begin{equation}
\label{14.5.2.0}
R_4 (t) = Q_2 \left ( \frac{Q_2^{\prime \prime}}{2t} - \frac{2 Q_2^\prime}{t} - \frac{5 Q_2^\prime}{4 t^2} \right )   
+ Q_2^2 \left ( \frac{2}{t} + \frac{5}{2 t^2} + \frac{3}{2 t^3} \right ) 
\end{equation}
Using (\ref{14.5R}) and (\ref{14.5.0}) in \eqref{14.5.2.0} we get, after some algebra,
\begin{equation}
\label{eqR3}
R_3 (t) =  J_0  - t I_0^2 -I_0^2 
\end{equation}
\begin{equation}
\label{eqR4}
R_4 (t) = \frac{t}{2} \left ( I_0^2 + I_0^3 - 2 I_0 J_0 \right ) + \frac{1}{2} J_0 - \frac{1}{4} I_0 - \frac{1}{2} I_0 
J_0 + \frac{1}{4} I_0^3 
\end{equation}
In the appendix (see equations \eqref{eqR3m} and \eqref{eqR4m}), it is shown
that
 $0 \le R_3 (t) \le R_{3,m} \le 0.02057$ 
and $0 \le R_4 (t) \le R_{4,m} \le 0.0009042 $ for $t \ge 1.99$.
Instead of using a variation of parameter formula for the third order linear operator on
the left of (\ref{14.4.0}) and 
turn (\ref{14.4.0}) into an integral equation, we find it 
convenient to define the auxiliary function
\begin{equation}
\label{14.5.1.1.1}
h(t) = e^{t} \left ( \sqrt{t} \mathcal{E}^{\prime \prime}
- \frac{\mathcal{E}^\prime}{2 \sqrt{t}} + \frac{\mathcal{E}(t)}{2 t^{3/2} } 
\right )
\end{equation}
and analyze the equation for $h$. 
\begin{Remark}
\label{remh}
\rm{
The choice $h$ is motivated by the observation
that 
\begin{equation}
\label{eqRelationh}
\frac{d^2}{dx^2} \sqrt{\frac{a}{2t (x)}} \mathcal{E} (t(x) ) =   
\sqrt{2} a^{3/2} e^{-t (x)} h(t (x)) 
\end{equation}
and thus by \eqref{eqRelationh}, \eqref{eqasymptq0} and  Lemmas \ref{LemN1} and \ref{LemN2},    $ g(x) = \sqrt{\frac{a}{2 t(x)}} 
q(t (x) ) \rightarrow 0$ implies
$h(t) \rightarrow 0$ as $t \rightarrow \infty$.
}
\end{Remark}

Equation 
(\ref{14.4}) can be rewritten as
\begin{equation}
\label{14.5.2}
h^\prime 
= - \frac{q_0 e^{t}}{2t} h + e^{t} B(t) \mathcal{E}   
-\frac{\mathcal{E} h}{2 t} 
-  t^{1/2} e^{t} R, 
\end{equation}
where
\begin{equation}
\label{14.8}
B (t) = - \frac{q_0^{\prime \prime} (t)}{2 t^{1/2} }
+ \frac{q_0^\prime (t)}{4 t^{3/2}} - \frac{q_0 (t)}{4 t^{5/2}} 
\end{equation}
The function $\mathcal{E}(t)$ can be written in terms of $h$ as follows:
\begin{equation}
\label{14.13}
\mathcal{E} (t) = t^{1/2}
\int_{\infty}^t \tau^{-1/2} \int_{\infty}^\tau 
s^{-1/2} e^{-s} h(s) ds d\tau 
\end{equation}
We write (\ref{14.5.2}) in integral form
\begin{equation}
\label{14.7}
h(t) = h_0 (t) - \int_{\infty}^t 
\frac{q_0 (\tau) e^{\tau} }{2 \tau} h(\tau) d\tau
+ \int_{\infty}^t e^{\tau} B (\tau) \mathcal{E} (\tau) 
d\tau
- \int_{\infty}^t \frac{h(\tau) \mathcal{E}(\tau)}{2 \tau} d\tau  
=: \mathcal{N} \left [ h \right ] (\tau),
\end{equation}
where
\begin{equation}
\label{14.9}
h_0 (t) = -\int_{\infty}^t e^{\tau} \tau^{1/2} R (\tau) d\tau.
\end{equation}
We will analyze \eqref{14.7} to find a unique exponentially decaying $h$ (cf. \eqref{eqNormH}), and then determine 
$\mathcal{E}$ from 
\eqref{14.13}.
\begin{Remark}{\rm 
The functions, $q_0, \mathcal{E}, h $ and $R$
depend on $c$ as well. For
simplicity, our notation we will not show this dependence, except when
needed.}
\end{Remark}
We will prove that 
the operator $\mathcal{N}$ defined in \eqref{14.7} is contractive in some small
ball in the space $\mathcal{H}$ defined as follows:
\begin{Definition}
Let $\mathcal{H}$ be the Banach space of continuous functions in $[T, \infty)$ equipped with
the weighted norm
\begin{equation}
\label{eqNormH}
\| h \| := \sup_{t \ge T} t e^{2 t} | h (t) | 
\end{equation}
\end{Definition}

\z We now prove some preliminary results needed in the proof of Proposition \ref{Prop3}.
\begin{Lemma}
\label{lemn2}
For $t$ in $[T, \infty)$ we have
\begin{equation}
\label{14.16}
\| h_0 \|
\le |c|^3 \left ( \frac{1}{2} R_{3,m} + \frac{|c| e^{-T}}{3 \sqrt{T}} R_{4,m} \right ) 
\le 
|c|^3 \left ( \frac{1}{2} R_{3,m} + \frac{|c| e^{-T}}{3 \sqrt{T}} 
R_{4,m} \right ) 
\end{equation}
\begin{equation}
\label{14.16.0}
\| \partial_c h_0 \|
\le c^2 \left ( \frac{3}{2} R_{3,m} + \frac{4 |c| e^{-T}}{3 
\sqrt{T}} R_{4,m} \right ) 
\end{equation}
where $R_{3,m}$ and $R_{4,m}$ are upper bounds for $|R_3 (t)|$ and
$|R_4 (t) |$ in $[T, \infty)$.
\end{Lemma}
\begin{proof}
Using \eqref{14.5R} and (\ref{14.9})
we obtain
\begin{multline}
\label{14.16.1}
\Big | h_0 (t) \Big | \le \int_{t}^\infty \left ( |c|^3 e^{-2 \tau} 
\tau^{-1}  R_{3,m} + 
c^4 e^{- 3 \tau} \tau^{-3/2} R_{4,m} \right ) d\tau  
\\
\le \frac{|c|^3 e^{-2 T}}{2 T} R_{3,m} + 
\frac{c^4 e^{-3 T}}{3 T^{3/2}} R_{4,m}
\end{multline}
implying (\ref{14.16})
After differentiating (\ref{14.9}) 
with respect to $c$, and using (\ref{14.5R}),
(\ref{14.16.0}) follows similarly.
\end{proof}
\begin{Lemma}
\label{lemn1}
\begin{equation}
\label{14.15}
\Big | \mathcal{E} (t) \Big | \le \frac{1}{9 t^{3/2}} e^{-3 t} \| h \|   \ , 
\Big | \partial_c \mathcal{E} (t) \Big | \le 
\frac{1}{9 t^{3/2}} e^{-3 t} \| \partial_c h \|  
\end{equation}
\end{Lemma}
\begin{proof} 
We note that
$$ \Big | \int_{\tau}^\infty s^{-1/2} e^{-s} h(s) ds  
\Big | \le \frac{1}{3} \tau^{-3/2} \| h \| e^{-3 \tau}$$
Using the inequality above it follows that
$$ \Big | t^{1/2} \int_t^{\infty} \tau^{-1/2} \int_{\tau}^\infty 
s^{-1/2} e^{-s} h(s) ds \Big |
\le \frac{1}{9} t^{-3/2} e^{-3t} \| h \| $$
The bounds on $\partial_c \mathcal{E}$ are obtained in a similar way since
(\ref{14.13}) implies
$$ \partial_c \mathcal{E}  (t; c) = 
t^{1/2} \int_{\infty}^t \tau^{-1/2} \int_{\infty}^t s^{-1/2} e^{-s} \partial_c h (s; c) ds d\tau $$
\end{proof}
Let
\begin{equation}
  \label{eq:defd0}
   d_0 = \frac{e^{-T}}{\sqrt{T}} \Big | Q_2 (T) + 0.0944 e^{-6.159 T} \Big |
\end{equation}
\begin{Lemma}
\label{lemq0}
For $t \ge T $, 
$q_0$ (cf. \eqref{14.3.0}) satisfies the following bounds
\begin{equation}
  \label{eq:L11.1}
  \Big | q_0 \Big | \le 
\frac{|c| e^{-t}}{\sqrt{t}} 
\left ( 1 + d_0 |c| \right )
\le 
\frac{|c| e^{-T}}{\sqrt{T}} 
\left ( 1 + d_0 |c| \right )
=:q_{0,m}
\end{equation}
\begin{equation} \label{eq:L11.2}
  \Big | \partial_c q_0 \Big | \le 
\frac{e^{-t}}{\sqrt{t}} 
\left ( 1 + 2 d_0 |c| 
\right )
\le 
\frac{e^{-T}}{\sqrt{T}} 
\left ( 1 + 2 d_0 |c| \right )
=:q_{0,c,m}
\end{equation}
\begin{equation} \label{eq:L11.3}
  \Big | q_0^\prime - \frac{q_0}{2t} \Big | 
\le \frac{|c| e^{-t}}{\sqrt{t}} \left ( 1 + \frac{3 |c| e^{-t}}{4 t^{3/2}} 
\right ) 
\le \frac{|c| e^{-T}}{\sqrt{T}} \left ( 1 + \frac{3 |c| e^{-T}}{4 T^{3/2}} 
\right ) =:q_{0,d,m} 
\end{equation}
\begin{equation}\label{eq:L11.4}
  \Big | \partial_c \left \{ q_0^\prime - \frac{q_0}{2t} \right \} \Big | 
\le \frac{e^{-t}}{\sqrt{t}} \left ( 1 + \frac{3 |c| e^{-t}}{2 t^{3/2}} 
\right ) 
\le \frac{e^{-T}}{\sqrt{T}} \left ( 1 + \frac{3 |c| e^{-T}}{2 T^{3/2}} 
\right ) =:q_{0,d,c,m} 
\end{equation}
\end{Lemma}
\begin{proof}
Using 
the integral representation for $q_0$, 
in the appendix  it is shown (see \eqref{A.q0}-\eqref{A.q0.2}) 
that
\begin{equation}
\label{eqq0M}
\Big | q_0 (t) \Big | \le 
\frac{|c| e^{-t}}{\sqrt{t}} 
\left \{ 1 + \frac{|c| e^{-t}}{2 \sqrt{T}} 
\left ( \Big | Q_2 (T) \Big | + 0.09044 e^{-s_0 T} \right ) \right \}
\end{equation}
\begin{equation}
\Big | \partial_c q_0 (t) \Big | \le 
\frac{e^{-t}}{\sqrt{t}} 
\left \{ 1 + \frac{|c| e^{-t}}{\sqrt{T}} 
\left ( | Q_2 (T) | + 0.09044 e^{-s_0 T} \right ) \right \}
\end{equation}
where $s_0 = 6.159\cdots$.
implying \eqref{eq:L11.1} and \eqref{eq:L11.2}.
Straightforward calculations show that 
$$ q_0^\prime (t) -\frac{q_0 (t)}{2t} = - \frac{c e^{-t}}{\sqrt{t}} (1 -I_0) 
- \frac{c^2 e^{-2t}}{2t} \left ( 1 -I_1 - \frac{I_1}{t} + \frac{I_1^2}{4 t^2} \right ) \, $$  
where $I_0$ is defined in (\ref{14.5}) and $I_1 = 2 t I_0$. The integral representation (\ref{14.5}) implies that $I_1: = 2 t I_0 \in (0, 1)$ and that
$ 1 - I_1 =\frac{3}{2 t} I_2$, where
$$ I_2 (t) = t \int_0^{\infty} \frac{e^{-st} ds}{(1+s)^{5/2}} $$  
and thus
$$ q_0^\prime (t) -\frac{q_0 (t)}{2t} = - \frac{c e^{-t}}{\sqrt{t}} (1 -I_0) 
- \frac{c^2 e^{-2t}}{4t^2} \left ( 3 I_2 -2 I_1 + \frac{I_1^2}{2t} \right ) \, $$  
From the equation above and its  $c$-derivative, 
and the fact that $I_0, I_1, I_2$ are in $(0, 1)$, \eqref{eq:L11.3} and \eqref{eq:L11.4} follow.
\end{proof}
\z With $d_0 $ given in \eqref{eq:defd0} let
\begin{equation}
\label{14.17.1}
d_q = \frac{|c|}{4} T^{-3/2} \left (1 + 
|c| d_0 \right );\  d_{q,c} =   
\frac{1}{4} T^{-3/2} \left ( 1 + 2 |c| d_0 \right )
\end{equation}

\begin{Lemma}
\label{lemq}
For $t \ge T \ge 1$,  we have
\begin{equation}
\label{14.17}
\left\| \int_{\infty}^t \frac{q_0 (\tau) e^{\tau}}{2 \tau} h (\tau) d\tau 
\right\|  
\le d_q  
\| h \|  \ , ~ 
\left\| \int_{\infty}^t \frac{\partial_c q_0 (\tau) e^{\tau}}{2 \tau} h (\tau) d\tau 
\right\|  
\le d_{q,c} \| h \| \ ,
\end{equation}
\end{Lemma}
\begin{proof}
The result follows from 
Lemma \ref{lemq0}, which implies
$$ 
\Big | \frac{e^{\tau} q_0 (\tau)}{2 \tau} \Big | \le \frac{c}{2 \tau^{3/2}} 
\left ( 1 + |c| d_0 \right ) $$ 
$$ 
\Big | \frac{e^{\tau} \partial_c q_0 (\tau)}{2 \tau} \Big | \le \frac{1}{2 \tau^{3/2}} 
\left ( 1 + 2 |c| d_0 \right ) $$ 
and noting that
$|h (t) | \le t^{-1}e^{-2t} \| h\|$.
\end{proof}
\z Let
\begin{equation}
\label{eqV}
V(t ; c) = - \frac{2}{c} t e^{t} B (t; c)
\end{equation}
\begin{Lemma}
\label{lemB}
$B$ defined in 
\eqref{14.8}
 satisfies the following inequalities
for $t \ge T \ge 1 $
\begin{equation}
\label{eqBm}
\Big | B(t; c) \Big | 
\le 
\frac{|c| e^{-t}}{2 t}  \left \{ 1 + \frac{3|c| e^{-t}}{4 t^{3/2}} \right \} 
\le \frac{|c| e^{-T}}{2 T}  \left \{ 1 + \frac{3|c| e^{-T}}{4 T^{3/2}} \right \} 
=:B_m 
\end{equation}
\begin{equation}
\label{eqBmc}
\Big |\partial_c B(t; c) \Big | \le 
\frac{e^{-t}}{2 t}  \left \{ 1 + \frac{3|c| e^{-t}}{2 t^{3/2}} \right \} 
\le \frac{e^{-T}}{2 T}  \left \{ 1 + \frac{3|c| e^{-T}}{2 T^{3/2}} \right \} 
=:B_{m,c} 
\end{equation}
\begin{equation}
\label{14.8.1.1.0.0.0}
\left | \partial_{t} \left [ \frac{2 t}{c} B(t; c )\right ]\right | 
\le 
e^{-t} \left ( 1 + \frac{c e^{-t}}{t^{3/2}} \right )
\le e^{-T} \left ( 1 + \frac{c e^{-T}}{T^{3/2}} \right )
=:B_{m,2,t} 
\end{equation}
\begin{equation}
\label{14.8.1.1.0.0}
\Big | \partial_{c} \left [ \frac{2 t}{c} B(t; c )\right ]\Big | 
\le \frac{3 e^{-2t}}{4 t^{3/2}} 
\le \frac{3 e^{-2T}}{4 T^{3/2}} 
=:B_{m,2,c} 
\end{equation}
$V(t; c)$ defined in \eqref{eqV} satisfies
\begin{equation}
\label{eqBmin}
V_{m} :=1 + \frac{3 c e^{-T}}{4 T^{3/2}} \ge 
 \Big | V(t; c) \Big |
\ge 1 - \frac{c e^{-T}}{4 T^{3/2} } 
=:V_{min}   
\end{equation}
\begin{equation}
\label{eqVp}
\Big | V^\prime (t; c) \Big | 
\le 
\frac{c e^{-T}}{2 T^{3/2}} =: V_{d,m} 
\end{equation}
\begin{equation}
\label{eqVc}
\Big | \partial_c V (t; c) \Big |
\le 
\frac{e^{-T}}{4 T^{3/2} } =:V_{c,m} 
\end{equation}

\end{Lemma}
\begin{proof}

Using  (\ref{14.3.0}) in (\ref{14.8}) it follows that
\begin{equation}
\label{Bexp}
B(t; c) = -\frac{c e^{-t}}{2 t}  - \frac{c^2 e^{-2t}}{8 t^{5/2}} \left ( 3 I_2 - I_1 \right ) 
\ , 
~~{\rm where} ~~I_1 = 2 t I_0 = t \int_0^\infty \frac{e^{-st}}{(1+s)^{3/2}} ds  \ , 
I_2 (t) = t \int_0^\infty \frac{e^{-st}ds}{(1+s)^{3/2}} 
\end{equation}
\begin{equation}
\partial_c B(t; c) = -\frac{e^{-t}}{2 t}  - \frac{c e^{-2t}}{4 t^{5/2}} \left ( 3 I_2 - I_1 \right ) 
\end{equation}
\begin{equation}
\partial_t \frac{2t B(t; c)}{c} = e^{-t} + \frac{c e^{-2t}}{4 t^{3/2}} \left ( 3 I_2 + I_1 \right ) 
\end{equation}
\begin{equation}
\partial_c \frac{2 t B (t; c)}{c} = -\frac{e^{-2t}}{4 t^{3/2}} (3 I_2 - I_1 ) 
\end{equation}
from which, again using the fact that $I_2 , I_1 \in (0, 1)$, \eqref{eqBm}--\eqref{14.8.1.1.0.0}.
To prove (\ref{eqBmin})--(\ref{eqVc}), 
note that
\begin{equation}
V (t;c) =: -\frac{2}{c} t e^{t} B (t; c)
= 1 + \frac{c e^{-t}}{4 t^{3/2}} \left ( 3 I_2 - I_1 \right ) \ ,   
\end{equation}
\begin{equation}
V^\prime (t; c) = -\frac{c e^{-t}}{2 t^{3/2}} I_1 (t) 
\end{equation}
\begin{equation}
\partial_c V (t; c) 
= 
\frac{e^{-t}}{4 t^{3/2}} \left ( 3 I_2 - I_1 \right ) 
\end{equation}
\end{proof}
\begin{Lemma}
\label{lemBterm}

For $T \ge 1 $ we have 
\begin{equation}
\label{14.18}
\L| \int_{\infty}^t e^{\tau} B(\tau) \mathcal{E}(\tau) d\tau 
\R|
\le d_B \| h \|,\ \ \ \  
\L| \int_{\infty}^t e^{\tau} \partial_c B(\tau) \mathcal{E}(\tau) d\tau 
\R| \le d_{B,c} \| h \| 
\end{equation}
where
\begin{equation}
d_B = 
\frac{|c| e^{-T}}{54 T^{3/2}}  
\left ( 1 + |c|  d_1 \right ),
 \  
d_{B,c} = \frac{e^{-T}}{54 T^{3/2}}  \left ( 1 + 2 |c| d_1 \right ),\ {\rm where} ~
d_1 =  
\frac{3 e^{-T}}{4 T^{3/2}} 
\end{equation}
\end{Lemma}
\begin{proof}
Using \eqref{14.15} \eqref{eqBm} 
and \eqref{eqBmc}
in Lemma \ref{lemB} the result follows immediately by integration.
\end{proof}
\begin{Lemma}
For $T > 0$ we have,
\begin{equation}
\label{14.19}
\L| \int_{t}^\infty \frac{h(\tau) 
\mathcal{E}(\tau)}{2 \tau} d\tau \R|
\le \frac{\| h \|^2}{90 T^{5/2}} 
e^{-3 T}   
\end{equation}
\end{Lemma}
\begin{proof}
Using Lemma \ref{lemn1} and the fact
$\Big | h (t) \Big | \le t^{-1} e^{-2t} \| h \|$ (which 
follows from the fact that $h\in\mathcal{H}$, the result follows by integration.
\end{proof}
\begin{Proposition}{\rm 
\label{Prop1}
For $|c| \le \frac{1}{4}$, $\epsilon =0.03 $ and $T \ge 1.99$,  
there exists a unique solution to the integral equation (\ref{14.7}) in a ball of size 
$ (1+\epsilon) \| h_0 \|$,  
implying that 
$\| h \| \le (1+\epsilon)\| h_0 \| 
\le 1.6667 \times 10^{-4} $. 
}\end{Proposition}
\begin{proof}
For 
$T \ge 1.99$, $|c| \le \frac{1}{4}$ and
$\epsilon= 0.03$, by Lemmas \ref{lemn2}, \ref{lemq} and  \ref{lemBterm}  and 
using $Q_2 (1.99)  = 0.0147\cdots$ and inequalities $R_{3,m}\le 0.0205666$, $R_{4,m} \le 0.009042$ 
for $T \ge 1.99$, see  (\ref{eqR3m}) and (\ref{eqR4m}) in
the appendix, we get
\begin{equation}
\label{eq1Prop1}
\| \mathcal{N} [h] \|
\le \| h_0 \| + (d_q+d_B ) (1+\epsilon) \| h_0 \|
+ \frac{e^{-3 T}}{90 T^{5/2}} (1+\epsilon)^2 \| h_0 \|^2 
\le (1+\epsilon) \| h_0 \|
\end{equation}
\begin{equation}
\label{eq2Prop1}
\| \mathcal{N} [h_1] - \mathcal{N} [h_2] \|
\le \left \{ d_q + d_B 
+ \frac{e^{-3 T}}{45 T^{5/2}} (1+\epsilon) \| h_0 \| \right \} 
\|h_1 - h_2 \|  
\end{equation}
implying contractivity of the integral operator in the stated ball.
\footnote{The values of the error function can be calculated using, for instance, \cite{Abramowitz}, 7.1.28.}
\end{proof}
\begin{Remark}\label{R7}{\rm 
For $|c| < \frac{1}{4} $ and $T \ge 1.99$, $h\in\mathcal{H}$ implies 
\begin{equation}
\label{hbound}
| h (t) | \le \| h \| t^{-1} e^{-2t} \le 1.6667 \times 10^{-4}  \times T^{-1} e^{-2 T} 
=: h_m \le 1.5651 \times 10^{-6};\ \forall t\ge T 
\end{equation}
}\end{Remark}
\begin{Remark}
\rm{
By uniqueness, this is the only solution
with $h \rightarrow 0$ as $t \rightarrow \infty$; we have proved that such a solution $h\in\mathcal{H}$ 
}\end{Remark}
\begin{Remark}\rm{  
Proposition \ref{Prop1} extends to 
any $c$ if $T$ is large enough, as seen  in the next proposition. This is likely to be useful in extending the present techniques
to more general initial conditions than (\ref{4})
}
\end{Remark}
\begin{Proposition}{\rm 
\label{Prop1ext}
For any $c$, there exists $T \ge 1$ large enough so that 
the integral equation (\ref{14.7}) has a unique solution 
$h \in \mathcal{H}$. 
}
\end{Proposition}
\begin{proof}
It is clear from Lemmas \ref{lemn2}, \ref{lemq}, \ref{lemBterm}
that for any given $c$, 
the functions $d_q$, $d_B$ and $\|h_0 \|$ are decreasing in $T$. 
Thus, the conditions (\ref{eq1Prop1})-(\ref{eq2Prop1}) are met for any fixed $\epsilon > 0$.
\end{proof}
\begin{Lemma}
For $0 < a \le a_r $, $|c| < \frac{1}{4}$ and $t \ge T \ge 1.99$, 
the function $\mathcal{E}$ (see \eqref{14.13})
satisfies following bounds
$$ 
\Big | 
\sqrt{\frac{a}{2t}}\mathcal{E} \Big | \le 
\sqrt{\frac{a_r}{2}} \frac{1}{9 t^2} e^{-3 t} \| h \|
\le 1.69 \times 10^{-5} t^{-2} e^{-3 t} $$     
$$ 
\Big | 
\frac{d}{dx} \sqrt{\frac{a}{2t}}\mathcal{E} \Big | \le 
\frac{a_r}{3}  e^{-3 t} t^{-3/2} \| h \| 
\le 9.20 \times 10^{-5} t^{-3/2} e^{-3 t}
$$
$$ \Big | 
\frac{d^2}{dx^2} \sqrt{\frac{a}{2t}}\mathcal{E} \Big | \le 
\sqrt{2} a_r^{3/2} \| h \| t^{-1} e^{-3t} \le    
5.02 \times 10^{-4} t^{-1} e^{-3t} 
$$
\end{Lemma}
\begin{proof}
From (\ref{14.13}), the first statement follows immediately.
The second statement follows from noting that the transformation 
(\ref{14.1}) implies that
$$ \frac{d}{dx} \sqrt{\frac{a}{2t}} \mathcal{E} =
a \left ( \mathcal{E}^\prime - \frac{1}{2t} \mathcal{E} \right )
= a \int_{\infty}^t \tau^{-1/2} e^{-\tau} h(\tau)  d\tau $$
Furthermore, we can check
$$ \frac{d^2}{dx^2} \sqrt{\frac{a}{2t (x)}} \mathcal{E} (t (x)) 
= \sqrt{2} a^{3/2} e^{-t} h, $$  
and hence the third statement follows.
\end{proof}
\begin{Lemma}
\label{lemhod}
The function $h$ satisfies
$$ \| h^\prime (\cdot; c) \| \le 
2 d_q \| h \| 
+ \frac{B_m}{9 T^{1/2}} \| h \| 
+ \frac{e^{-3 T}}{18 T^{5/2}} \| h \|^2
+|c|^3 \left ( R_{3,m}   
+\frac{|c| e^{-T}}{T^{1/2}} R_{4,m} \right )
$$
\end{Lemma}
\begin{proof}
We note from 
\eqref{14.5.2}
$$ \| h^\prime \| \le \sup_{t \ge T} \frac{|q_0 | e^{t}}{2t} \| h \|
+ \sup_{t \ge T} \frac{1}{9 t} \Big | B (t) \Big |  \| h \| 
+ \frac{e^{-3 T}}{18 T^{5/2}} \| h \|^2  
|c|^3 R_{3,m}   
+c^4 \frac{e^{-t}}{t^{1/2} } R_{4,m}   
$$
Using Lemma \ref{lemq} and and \ref{lemB} we get the statement follows.
\end{proof}
\begin{Lemma} 
\label{lemhcb}
The function $h$ satisfies
$$ \| \partial_c h ( \cdot , c) \|
\le 
\left (1 -  
d_q - d_B - \frac{e^{-3 T}}{45 T^{5/2}} \| h \| \right )^{-1} 
\left \{ \| \partial_c h_0 \| + \left ( d_{q,c} + d_{B, c} \right ) \|h \|_\infty \right \}
$$
\end{Lemma} 
\begin{proof}
We note that (\ref{14.7}) implies
\begin{multline}
\label{22}
\partial_{c} h (t; c) = \partial_c h_0 (t; c) 
- \int_{\infty}^t \frac{e^{\tau}}{2 \tau} \partial_c q_0 (\tau; c) h(\tau; c)  
+ \int_{\infty}^t e^{\tau} \partial_c B (\tau; c) E (\tau; c) \\ 
- \int_{\infty}^t \frac{e^{\tau}}{2 \tau} q_0 (\tau; c) \partial_c h(\tau; c) d\tau  
- \int_{\infty}^t e^{\tau} B (\tau; c) \partial_c E (\tau; c) d\tau \\
- \int_{\infty}^t \partial_c h (\tau; c) E (\tau; c) \frac{d\tau}{2 \tau} 
- \int_{\infty}^t h (\tau; c) \partial_c E (\tau; c) \frac{d\tau}{2 \tau} 
\end{multline}
Applying Lemmas \ref{lemn1}, \ref{lemn2}, 
\ref{lemq}
and 
\ref{lemB}
to  (\ref{22}) we get
\begin{equation}
\| \partial_c h (t; c) \|
\le \| \partial_c h_0 \| + (d_{q,c} + d_{B, c} ) \| h \| 
+ \left (d_q + d_B + \frac{e^{-3 T}}{45 T^{5/2}} \| h \| \right )  
\| \partial_c h (\cdot, c) \|_{\infty}     
\end{equation}
\end{proof}
\begin{Remark}
Since $h\in\mathcal{H}$ Lemmas \ref{lemhod} and
\ref{lemhcb} imply
\begin{equation}
\label{eqdm}
\Big | h^\prime (T; c) \Big | \le 
T^{-1} e^{-2 T} \left \{ 2 d_q \| h \| 
+ \frac{B_m}{9} \| h \| 
+ \frac{e^{- 3 T}}{18 T^{5/2}} \| h \|^2
+|c|^3 \left ( R_{3,m}   
+\frac{|c| e^{-T}}{T^{1/2}} R_{4,m} \right )
\right \}
=:h_{d,m}
\end{equation}
\begin{equation}
\label{eqhcm}
\Big | \partial_c h (T; c)  \Big | 
\le T^{-1} e^{-2 T} 
\left (1 -  
d_q - d_B - \frac{e^{-3 T}}{45 T^{5/2}} \| h \| \right )^{-1} 
\left \{ \| \partial_c h_0 \| + \left ( d_{q,c} + d_{B, c} \right ) 
\|h \|_\infty \right \}
=: h_{c,m}
\end{equation}
\end{Remark}

\subsection{End of proof of Proposition \ref{Prop3}}

For given $|c| < \frac{1}{4}$ and $T \ge 1.99$,
Proposition \ref{Prop1} implies that
$\mathcal{E}$ satisfies (\ref{14.4}) for $t \ge T \ge 1.99$. 
This implies the existence of a solution $q=q_0 +\mathcal{E}$ 
satisfying  (\ref{14.2})
and having $t^{-1/2} e^{-t}$ decay as $t \rightarrow \infty$. 
From Lemmas \ref{LemN1} and \ref{LemN2},
this is the only 
solution for which $\frac{q}{\sqrt{t}} \rightarrow 0$ 
as $t \rightarrow \infty$,
Proposition \ref{Prop3} follows 
since the transformation (\ref{14.1}) for $a >0$ in the regime $t \ge T 
\ge 1.99$ 
guarantees that $F$ will satisfy
(\ref{3}) for $x \ge \frac{5}{2}$. 

\section{Matching and proof of Proposition \ref{Prop4}}
\label{S3}

In order for the two representations \eqref{12.1.0.0} and \eqref{14.1} to coincide at $x = \frac{5}{2}$ 
we match $F$ and its two derivatives; from (\ref{14.3.0}), (\ref{14.3.1}) and (\ref{14.5.1.1.1})  we get
\begin{equation}
\label{16}
a = F^\prime (\tfrac{5}{2}) 
- a \left ( q_0^\prime (t_m; c) - \frac{q_0 (t_m; c)}{2t_m} \right ) -  
a \int_{\infty}^{t_m} \frac{e^{-\tau}}{\sqrt{\tau}} h (\tau; c) d\tau 
=:N_1 (a, b, c)
\end{equation}
\begin{equation}
\label{15}
b = F(\tfrac{5}{2}) - \frac{5}{2} N_1 - \sqrt{\frac{a}{2 t_m}} q_0 (t_m; c)   
- \sqrt{\frac{a}{2}} 
\int_{\infty}^{t_m} \tau^{-1/2} \int_{\infty}^\tau s^{-1/2} e^{-s} h(s; c) ds 
:=N_2 (a, b, c)
\end{equation}
\begin{equation}
\label{17}
c =  
\frac{1}{\sqrt{2} a^{3/2}}
\left [ V (t_m; c) + \frac{1}{c} h (t_m; c) \right ]^{-1}  
e^{t_m} F^{\prime \prime} (\tfrac{5}{2}) =: N_3 (a, b, c) 
\end{equation}      
\begin{Definition}
We define ${\bf A} = \left ( a , \frac{b}{2}, \frac{c}{2} \right )$, ${\bf A}_0 = \left ( a_0, \frac{b_0}{2}, \frac{c_0}{2} \right )$
and ${\bf N} ({\bf A} ) = \left ( N_1, \frac{1}{2} N_2, 
\frac{1}{2} N_3 \right )$. Define also
$$\mathcal{S}_A := \left \{ 
\| {\bf A} - {\bf A}_0 \|_2 \le \rho_0:=5 \times 10^{-5} \right \}$$
where $\|. \|_2$ is the Euclidean norm and let
\begin{equation}
\label{17.3.0}
{\bf J} = {\begin{pmatrix} \partial_a N_1 & 2 \partial_b N_1  & 2 \partial_c N_1 \cr
         \frac{1}{2} \partial_a N_2 &  \partial_b N_2 & \partial_c N_2 \cr
         \frac{1}{2} \partial_a N_3 &  \partial_b N_3 & \partial_c N_3 
\end{pmatrix}}
\end{equation}
\end{Definition}
\begin{Note}{\rm }
  We see that ${\bf A} \in \mathcal{S}_A $ implies $(a, b, c) \in \mathcal{S}$.  
The system of equations (\ref{16})-(\ref{17}) is written as
\begin{equation}
\label{15.0.0}
{\bf A} = \mathbf{N} [{\bf A} ]
\end{equation}
We define
${\bf J} = \frac{\partial{\bf N}}{\partial {\bf A}}$ to be
the Jacobian and $\|{\bf J} \|_2$ denotes the
$l^2 $ (Euclidean) norm of the matrix. We note that
\begin{multline}
\label{17.3.0.0}
\| J \|^2_2 = 
\left ( \partial_a N_1 \right )^2 + 4 \left (\partial_b N_1\right )^2  + 4 \left ( \partial_c N_1 \right )^2
+\frac{1}{4} \left ( \partial_a N_2 \right )^2 
+\left (\partial_b N_2\right )^2  + \left ( \partial_c N_2 \right )^2
\\
+\frac{1}{4} \left ( \partial_a N_3 \right )^2 
+\left (\partial_b N_3\right )^2  + \left ( \partial_c N_3 \right )^2
\end{multline}
\end{Note}

\begin{Lemma}
\label{lemMatch}{\rm 
The inequalities
\begin{equation} \label{162}
\| {\bf A}_0 - {\bf N} [{\bf A}_0] \|_2 \le (1-\alpha) \rho_0
\end{equation}
\begin{equation}\label{163}
\sup_{{\bf A} \in \mathcal{S}_A} \| {\bf J} \|_2 \le \alpha < 1
\end{equation}
for some $\alpha \in (0, 1)$ imply that
${\bf A} = \mathbf{N} [ {\bf A} ]$
has a unique solution for ${\bf A} \in \mathcal{S}_A$.}
\end{Lemma}
\begin{proof}
The mean-value theorem implies 
\begin{multline}
\| {\bf N} [{\bf A} ] - {\bf A}_0 \|_2 
\le \| {\bf N} [{\bf A}_0] - {\bf A}_0 \|_2 + 
\| {\bf N} [{\bf A}] - {\bf N} [{\bf A}_0 ] \|_2  
\le \rho_0 (1-\alpha) + \| {\bf J} \|_2 \rho_0   
\le \rho_0 
\end{multline}
and also, if ${\bf A}_1, {\bf A}_2 \in \mathcal{S}_A$:
$$
\| {\bf N} [{\bf A}_1 ] 
- {\bf N} [{\bf A}_2 ] \|_2 
\le \|{\bf J} \|_2 \| {\bf A}_1 - {\bf A}_2 \|_2  
\le \alpha \|{\bf A}_1 - {\bf A}_2 \|_2 
$$
Thus, \eqref{162} and \eqref{163} imply that $\mathbf{N}: \mathcal{S}_A 
\rightarrow \mathcal{S}_A$ and that it is contractive there; the result follows from the contractive mapping theorem. 
\end{proof}
\subsection{Proof  of Proposition \ref{Prop4} }
Proposition \ref{Prop4} follows from Lemma \ref{lemMatch}
once we show that \eqref{162} and \eqref{163} hold. 
In the following two subsections 
it will be shown that
$\alpha \le 0.764$ and that $ \| {\bf A}_0 - \mathbf{N} [{\bf A}_0] \|_2 
\le 1.16 \times 10^{-5} \le (1-\alpha) \rho_0$ thereby completing
the proof of Proposition
\ref{Prop4}.

\begin{Remark} 
\label{remCont}
\rm{
Note that the proof of Proposition \ref{Prop4}  only requires smallness of the norms of $h$ and $E$  (we recall that  $F=F_0+E$) and on no further details about them. If in some application  $F_0$ needs to be made $C^2$, then this can be ensured by iterating $\mathbf{N}$ with $h=E=0$;  the
first thirteen digits obtained in this way are given in (\ref{eqabcCont}).

}
\end{Remark}

\subsection{Bounds on
$ \| \mathbf{N} ({\bf A}_0 ) - {\bf A}_0 \|$}

We note that 
\begin{multline}
\label{17.1}
\Big | a_0 - N_1 (a_0, b_0, c_0) \Big | \le   
\Big | a_0 - F_0^\prime (\tfrac52) + a_0 \left (q_0^\prime (t_{m,0}; c_0) - 
\frac{q_0 (t_{m,0}; c_0)}{2 t_{m,0}} \right ) \Big | \\
+ | E^{\prime} (\tfrac52) |   
+\frac{a_0}{3} t_{m,0}^{-3/2} e^{-3 t_{m,0}} \| h \|
\le 4.81 \times 10^{-6}
\end{multline}
\begin{multline}
\label{17.2}
\Big | b_0 - N_2 (a_0, b_0, c_0) \Big |
\le \Big | b_0 - F_0 (\tfrac52) + \tfrac52 a_0  
+ \sqrt{\frac{a_0}{2 t_{m,0}}}  q_0 (t_{m,0}; c_0) \Big | + \frac{5}{2}
\Big | a_0 - N_1 (a_0, b_0, c_0)
\Big | \\ 
+ \Big | E (\tfrac52) \Big | +\frac{1}{9} \sqrt{\frac{a_0}{2}} t_{m,0}^{-2} 
e^{-3 t_{m,0}} \|h\| 
\le 1.64 \times 10^{-5} 
\end{multline}
\begin{multline}
\label{17.3}
\Big | c_0 - N_3 (a_0, b_0, c_0) \Big |
\le \Big | c_0 - \frac{1}{\sqrt{2} a_0^{3/2} } \left ( V (t_{m,0}, c_0 ) \right )^{-1} 
e^{t_{m,0}} F_0^{\prime \prime} (\tfrac52) \Big | \\  
+ \frac{F_0^{\prime \prime} (\tfrac52)}{
\sqrt{2} c_0 a_0^{3/2}} t_{m,0}^{-1} e^{-t_{m,0}} \| h \|
\left ( V (t_{m,0}; c_0)  - \frac{e^{-2 t_{m,0}}}{c_0 t_{m,0}} \| h \| \right )^{-1}
\left ( V (t_{m,0} \right )^{-1} \\
+ 
\frac{e^{t_{m,0}}}{\sqrt{2} a_0^{3/2}} \Big | E^{\prime \prime} (\tfrac52) \Big |
\left ( V (t_{m,0}; c_0)   
- \frac{ e^{-2 t_{m,0}} \|h \|}{c_0 t_{m,0}} \right )^{-1}   
\le 1.33 \times 10^{-5}
\end{multline} 
implying
\begin{equation}
\| {\bf A}_0 - {\bf N} ({\bf A}_0 ) \|_2 
\le 1.16 \times 10^{-5}
\end{equation}

\subsection{Bounds on the derivatives of $N_j$ and 
on $\| {\bf J} \|_2$}
We note that 
\begin{multline}
\label{20}
\partial_{a} N_1 = 
-\left ( q_0^\prime (t_m; c) - \frac{q_0 (t_m; c)}{2 t_m} \right )
+a \left ( \frac{25}{4} - \frac{b^2}{a^2} \right ) t_m^{1/2} B(t_m; c) \\
- \int_{\infty}^{t_m} \frac{e^{-\tau}}{\tau^{1/2}}  h(\tau; c) d\tau 
- \frac{a}{2t_m^{1/2}} \left ( \frac{25}{4}
- \frac{b^2}{a^2} \right ) e^{-t_m} h (t_m) 
\end{multline}
Remark \ref{R7} implies that the last two terms on the rhs of \eqref{20} are bounded by
$t_{m,l}^{-1/2} e^{-t_{m,l}} h_m \left [ 1 + \frac{a_r}{2} \left (\frac{25}{4} 
- \frac{b_l^2}{a_r^2} \right ) \right ]$. 
Applying now Lemmas \ref{lemq0}, \ref{lemB} to (\ref{20}) we get
\begin{multline}
\label{20.1}
\Big | \partial_a N_1 \Big | 
\le \Big | - q_0^\prime (t_{m,0}, c_0) + \frac{q_0 (t_{m,0}, c_0)}{2 t_{m,0}} 
+a_0 \left ( \frac{25}{4} - \frac{b_0^2}{a_0^2} \right ) t_{m,0}^{1/2} B (t_{m,0}; c_0) \Big |  \\
+ \left \{ \frac{25}{4} (a_r-a_0) + \left ( \frac{b_r^2}{a_l} - \frac{b_0^2}{a_0} \right ) \right \} 
t_{m,r}^{1/2} B_m 
+ \frac{e^{-t_{m,l}}}{t_{m,l}^{1/2}} h_m \left \{ 1 + \frac{a_r}{2} \left [ 
\frac{25}{4} - 
\left ( \frac{b_l^2}{a_r^2} \right ) \right ] \right \}  \\
+ \left \{ 2 t_{m,r}^{1/2} B_m + a_0 \left ( \frac{25}{4} - \frac{b_0^2}{a_0^2} \right ) 
\left ( \frac{B_m}{2 t_{m,l}^{1/2}} + \frac{c_r}{2 t_{m,l}^{1/2}} B_{m,2,t} \right ) \right \} (t_{m,r}-t_{m,0})       
\\
+ \left ( q_{0,d,c,m} + a_0 \left ( \frac{25}{4} - \frac{b_0^2}{a_0^2} \right ) \sqrt{t_{m,r}} B_{m,c} \right ) 
(c_r - c_0) 
\le 0.081  \ , 
\end{multline}
The maximal value of the bounds is attained when $c=c_r$ and $T=t_{m,l}$, which we used to get the results above.
\begin{equation}
\label{21}
\partial_{b} N_1 = \sqrt{2 a} \left [ 2 t_m B (t_m; c) - e^{-t_m} h(t_m; c) 
\right]
\end{equation}
hence using \eqref{hbound} and Lemma \ref{lemB},
\begin{equation}
\label{21.0.1}
\Big | \partial_b N_1 \Big | \le \sqrt{2 a_r} \left ( 2 t_{m,r} B_m + e^{-t_{m,l}} h_m \right ) 
\le 0.059 
\end{equation}
\begin{equation}
\label{21.1}
\partial_c N_1 
= -a \partial_c \left \{ q_0^\prime (t_m; c) - \frac{q_0 (t_m)}{2 t_m} \right \}
-a \int_{\infty}^{t_m} \frac{e^{-\tau}}{\sqrt{\tau}} \partial_c h (\tau; c) d\tau  
\end{equation}
implying from 
Lemmas \ref{lemq0} and equation
\eqref{eqhcm},
\begin{equation}
\label{21.2}
\Big | \partial_c N_1 \Big | 
\le  
a_r q_{0,d,c,m}
+ a_r e^{-t_{m,l}} t_{m,l}^{-1/2} h_{c,m} 
\le 0.163 
\end{equation}
We now consider 
\begin{multline}
\label{21.3}
\partial_a N_2 = -\frac{5}{2} \partial_a N_1  
-\frac{1}{2} \sqrt{\frac{1}{2 a t_m}} q_0(t_m)  
-\frac{1}{2} \sqrt{\frac{a}{2 t_m}} \left ( \frac{25}{4}
-\frac{b^2}{a^2} \right ) 
\left ( q_0^\prime (t_m) - \frac{q_0 (t_m)}{2 t_m} \right ) \\
-\frac{1}{2} \sqrt{\frac{1}{2 a t_m}} \mathcal{E} (t_m)  
-\frac{1}{2} \sqrt{\frac{a}{2 t_m}} \left ( \frac{25}{4}-\frac{b^2}{a^2} \right ) 
\int_{\infty}^{t_m} \frac{e^{-\tau} h(\tau)}{\sqrt{\tau}} d\tau  
\end{multline}
It follows from Lemmas \ref{lemn2}, \ref{lemq0}, \ref{lemB},
Proposition \ref{Prop1},  and equations \eqref{hbound} and \eqref{eqhcm}
that
\begin{multline}
\label{21.4}
\Big | \partial_a N_2 \Big | \le \frac{5}{2} \Big | \partial_a N_1 \Big |  
+\frac{1}{2} \sqrt{\frac{1}{2 a_l t_{m,l}}} \Big | q_0 (t_{m,0}; c_0) 
+ a_0 \left ( \frac{25}{4} - \frac{b_0^2}{a_0^2} \right ) 
\left ( q_0^\prime (t_{m,0}; c_0) - \frac{q_0 (t_{m,0}; c_0)}{2 t_{m,0}} \right ) \Big | \\
+ \frac{1}{2} \sqrt{\frac{1}{2a_l t_{m,l}}} \left \{ \frac{25}{4} (a_r-a_0) + \left ( \frac{b_r^2}{a_l} - 
\frac{b_0^2}{a_0} \right )
\right \} q_{0,d,m} + 
\frac{e^{-t_{m,l}}}{2^{3/2} a_l^{1/2} t_{m,l}}  
\left ( 1 + a_r \left [\frac{25}{4}-\frac{b_l^2}{a_r^2} \right ] \right ) h_m \\
+\frac{1}{2} \sqrt{\frac{1}{2 a_l t_{m,l}}} \left \{ q_{0,d,m} + \frac{q_{0,m}}{2 t_{m,l}}   
+ a_0 \left ( \frac{25}{4} - \frac{b_0^2}{a_0^2} \right ) 2 t_{m,r}^{1/2} B_{m} \right \} (t_{m,r}-t_{m,0}) \\
+\frac{1}{2} \sqrt{\frac{1}{2 a_l t_{m,l}}} \left \{ q_{0,c,m} 
+ a_0 \left ( \frac{25}{4} - \frac{b_0^2}{a_0^2} \right ) q_{0,d,c,m} \right \} (c_r-c_0) 
\le 0.232
\end{multline}
We also note that
\begin{equation}
\label{21.5.0}
\partial_b N_2 = - \frac{5}{2} \partial_b N_1 - 
\left ( q_0^{\prime} (t_m; c) - \frac{q_0 (t_m; c)}{2 t_m} \right ) 
- \int_{\infty}^{t_m} \frac{d\tau}{\sqrt{\tau}} e^{-\tau} h (\tau; c) , 
\end{equation}
and therefore Lemma \ref{lemq0} and (\ref{hbound}) imply
\begin{equation}
\label{21.5}
\Big | \partial_b N_2 \Big | \le \frac{5}{2} \Big | \partial_b N_1 \Big | + 
q_{0,d,m} + e^{-t_{m,l}} t_{m,l}^{-1/2} h_{m}
\le 0.168
\end{equation}
Now,
\begin{equation}
\label{21.6}
\partial_c N_2 = - \frac{5}{2} \partial_c N_1 - 
\sqrt{\frac{a}{2 t_m}} \partial_c q_0 (t_m; c)
-\sqrt{\frac{a}{2}} \int_{\infty}^{t_m} \tau^{-1/2} \int_{\infty}^\tau s^{-1/2} e^{-s} 
\partial_c h (s; c) ds d\tau 
\end{equation}
Hence, (\ref{eqhcm}) and  Lemma \ref{lemq0} imply 
\begin{equation}
\label{21.7}
\Big | 
\partial_c N_2 \Big | \le \frac{5}{2} \Big | \partial_c N_1 \Big | +
\sqrt{\frac{a_r}{2t_{m,l}}} q_{0,c,m}
+\sqrt{\frac{a_r}{2}} t_{m,l}^{-1} h_{c,m} e^{-t_{m,l}}  
\le 0.468
\end{equation}
Also,
\begin{multline}
\label{21.8.0}
\partial_a N_3 = 
\frac{1}{2 \sqrt{2} a^{3/2}} 
\left [ V(t_m; c) + \frac{1}{c} h (t_m; c)
\right ]^{-1} e^{t_m} F^{\prime \prime} (\frac{5}{2})
\left \{ - \frac{3}{a} + \left ( \frac{25}{4} - \frac{b^2}{a^2} \right ) \right .
\\
\left . \times \left ( 1 + \left [ V (t_m; c) + \frac{1}{c} h (t_m; c) \right ]^{-1}
\left [ V^\prime (t_m; c) + \frac{1}{c} h^\prime (t_m; c) \right ] \right )  
\right \}
\end{multline}
Therefore, from Lemma \ref{lemB}, equations \eqref{hbound}, \eqref{eqdm} and the positivity of 
$V_{min} - \frac{h_m}{c_l}$, $F_0^{\prime \prime}$ and 
$\frac{25}{4} - \frac{b_l^2}{a_r^2} - \frac{3}{a_l}$ (see 
Definition \ref{Def1}, \eqref{12.1.0}-\eqref{12.1.1} and \eqref{eqBmin}); in \eqref{eqBmin}
we used $c =c_r$, $T=t_{m,l}$, which minimize $V_m$).
It follows
that
\begin{multline}
\label{21.8}
\Big | \partial_a N_3 \Big | \le 
\frac{1}{2 \sqrt{2} a_l^{3/2}}  
\left [ V_{min} - \frac{h_m}{c_l} \right ]^{-1} 
e^{t_{m,r}} \left ( F_0^{\prime \prime} (\tfrac52) + \Big | E^{\prime \prime} (\tfrac52) \Big | \right )
\left \{ \frac{25}{4} - \frac{b_l^2}{a_r^2} - \frac{3}{a_l} 
+ \left ( \frac{25}{4} - \frac{b_l^2}{a_r^2} \right ) \right . \\ 
\left . \times 
\left [ V_{min} - \frac{h_m}{c_l} \right ]^{-1} \left [
V_{d,m} + \frac{1}{c_l} h_{d,m} \right ] 
\right \}
\le 0.44
\end{multline}
Further,
\begin{multline}
\label{21.9}
\partial_b N_3 = 
\frac{\sqrt{t_m}}{a^{2}} e^{t_m} F^{\prime \prime} (\tfrac52)
\left [ V(t_m; c) + \frac{1}{c} h (t_m; c) \right ]^{-2} \left [ V (t_{m;c} ) + \frac{1}{c} h (t_m; c)
- V^\prime (t_m; c) - \frac{1}{c} h^\prime (t_m; c) \right ] 
\end{multline}
Using again Lemma \ref{lemB}, and  equations \eqref{hbound},  \eqref{eqdm} we get
\begin{multline}
\label{21.10}
\Big | \partial_b N_3 \Big | \le 
\frac{\sqrt{t_m}}{a_l^{2}} e^{t_{m,r}} \left ( F_0^{\prime \prime} (\tfrac52) + \Big | E^{\prime \prime} (\tfrac52 ) \Big | 
\right ) 
\left [ V_{min} - \frac{h_m}{c} \right ]^{-2} 
\times \left \{ V_m + \frac{1}{c} h_m + V_{d,m} + \frac{1}{c} h_{d,m} \right \}
\le 0.384
\end{multline}
Furthermore, 
\begin{equation}
\label{21.11}
\partial_c N_3 = 
-\frac{1}{\sqrt{2} a^{3/2}} e^{t_m} F^{\prime \prime} (\tfrac52) 
\left [ V (t_m; c) + \frac{1}{c} h \right ]^{-2} 
\left [ \partial_c V (t_m; c) + \frac{1}{c} \partial_c h (t_m; c) - \frac{1}{c^2} h (t_m; c) \right ] 
\end{equation}
Lemma \ref{lemB} and  equations \eqref{hbound},  \eqref{eqhcm} imply
\begin{multline}
\label{21.12}
\Big | \partial_c N_3 \Big | \le  
\frac{e^{t_{m,r}}}{\sqrt{2} a_l^{3/2}} \left ( F_0^{\prime \prime} (\tfrac52) + \Big | E^{\prime \prime} (\tfrac52 ) \Big | 
\right )
\left [ V_{min} - \frac{h_m}{c} \right ]^{-2} 
\left [ V_{c,m} + \frac{1}{c_l} h_{c,m} + \frac{1}{c_l^2} h_m \right ]
\le 0.0029
\end{multline}
By straightforward calculations we get $\| {\bf J} \|_2 \le 0.764$.

\section{Appendix}

\subsection{Bounds on $q_0$ and $\partial_c q_0$}

Using the integral representations of $I_0$ and $J_0$,  (\ref{14.3.0}), (\ref{14.5}) and (\ref{14.5.0}) imply
\begin{equation}
\label{A.0}
q_0 (t)  = \frac{c}{\sqrt{t}} e^{-t} t \int_0^\infty \frac{e^{-st}}{(1+s)^{3/2}} ds 
+ c^2 e^{-2 t}  \int_0^\infty e^{-st} U(s) ds, 
\end{equation}  
where
\begin{equation}
\label{14.3.0.0.1}
U (s) = \frac{1}{2 (1+s/2)^{3/2}} - \frac{1}{2 (1+s)^{3/2} } 
- \frac{s}{(2+s)^2 \sqrt{1+s}} 
\end{equation}
Note that
\begin{equation}
\label{14.3}
\frac{Q_2 (t)}{t} = \int_0^\infty e^{-st} U (s) ds 
\end{equation}  
Making the change of variable
\begin{equation}
s = - 1 + \frac{(y-1)^2}{4y}
\end{equation}
which maps one-to-one  $ (0, \infty)$ onto $(3+2 \sqrt{2}, \infty)$ we obtain
from \eqref{14.3.0.0.1}
\begin{equation}
\label{A.2}
U (s(y))= \frac{4 y^{3/2} (3 -2 \sqrt{2} ) (y-3-2\sqrt{2}}{(y-1)^3 (1+y)^4} 
P_3 (y)
\end{equation}
where 
\begin{equation}
\label{A.3}
P_3 (y) = -y^3 + (17+10\sqrt{2} ) y^2 - (11+4 \sqrt{2} ) y + 3 + 2 \sqrt{2} 
\end{equation}
This cubic has only one real root $y_0 = 30.604 \cdots $ implying
$s = s_0 = 6.159 \cdots$. Since $P(y) >  0$ for $y < y_0$ and $P(y) < 0$ for
$y > y_0$,
$s_0 > s > 0$ corresponds to $y_0 > y > 3 + 2 \sqrt{2}$ and $s > s_0$ corresponds to $y> y_0$,
it follows from (\ref{A.2}) that
\begin{equation}
\label{eqUsign}
U(s)  >  0 ~~{\rm in} ~ (0, s_0) ~~{\rm and} ~~  U(s) < 0 ~~{\rm in }  (s_0, \infty) 
\end{equation}
Further, for $s > s_0$, ({\it i.e.} 
$y > y_0 \approx 30.604.. $), the functions
$$y^{-3} P_3 (y),\ ,\frac{y^{1/2}}{(y-1)}\ \text{and } \frac{4 y^4}{(y-3-2\sqrt{2} )(y-1)^2 (1+y)^4}$$
are  decreasing.
Therefore, it follows from (\ref{A.2}) that
for $s \ge s_0$, ($y > y_0)$ 
\begin{equation}
\label{A.4}
0 > U (s(y)) \ge -\frac{4 y^{9/2} (3 -2 \sqrt{2} ) (y-3-2\sqrt{2}}{(y-1)^3 (1+y)^4} 
\ge -\frac{4 y_0^{9/2} (3-2 \sqrt{2}) (y_0-3-2\sqrt{2})}{(y_0-1)^3 (1+y_0)^4}    
=-0.09437.. 
\end{equation}
Therefore, 
\begin{equation}
0 \ge \int_{s_0}^\infty U(s) e^{-s t} ds > -\frac{0.0944}{t} e^{-s_0 t}   
\end{equation}
Since $U(s)$ is being positive on $(0, s_0 )$ (see \eqref{eqUsign})
\begin{equation}
\int_0^{s_0} U(s) e^{-st} ds 
\end{equation} 
is clearly a decreasing positive function of $t$ 
for $t \ge T$ at thus attains its maximum at 
$t=T$.
Therefore
for $t \ge T$ we have
\begin{equation}
\int_0^\infty U(s) e^{-st} dt \le \int_0^{s_0} U(s) e^{-s T} ds  
\le \int_0^\infty U(s) e^{-s T} ds 
+ \frac{0.0944}{T} e^{-s_0 T}
\end{equation}
We conclude that
\begin{equation}
\label{A.q0}
\Big | q_0 (t) \Big |  \le \frac{|c| e^{-t}}{\sqrt{t}} +
c^2 e^{-2 t} \frac{Q_2 (T)}{T} +     
\frac{c^2}{T} e^{-2 t} 
0.0944 e^{-s_0 T}
\end{equation}
We note that that since $s_0 = 6.159 \cdots$, for $T \ge 1.99$,
$0.0944 e^{-s_0 T} \le 4.5 \times 10^{-7}$.
The numerical value $Q_2 (1.99) =0.0147 \cdots$ can be easily obtained from rigorous formulas 
(e.g. \cite{Abramowitz} 7.1.28) 
By differentiating $q_0$ with respect to $c$ we get in a similar way
\begin{equation}
\label{A.q0.2}
\Big | \partial_c q_0 (t) \Big |  \le \frac{e^{-t}}{\sqrt{t}} +
\frac{2 |c| e^{-2 t}}{T} Q_2 (T) +     
2 |c| \frac{e^{-2 t}}{T}  
0.0944 e^{-s_0 T}
\end{equation}

\subsection{Bounds on $R_3$ for $t \ge T \ge 1.99$}

In the formula (\ref{eqR3}) for $R_3$,  $I_0$ and $J_0$ have integral representations, see
(\ref{14.5}) and (\ref{14.5.0}). Using these representations we obtain
\begin{equation}
\label{A.3.0}
-I_0^2 (t) = -\frac{1}{4} \int_0^\infty e^{-st}  \int_0^s \frac{d\tau}{(1+\tau)^{3/2} (1+(s-\tau))^{3/2}} 
= -\int_0^\infty \frac{s e^{-st}}{(s+2)^2 \sqrt{1+s}} ds 
\end{equation}
\begin{equation}
\label{A.4.0}
-t I_0^2 = -\int_0^\infty e^{-st} \partial_s \left ( \frac{s}{(s+2)^2 \sqrt{1+s}} \right ) ds 
= \int_0^\infty e^{-st} \frac{3 s^2 - 4}{2 (s+2)^3 (s+1)^{3/2}}  ds  
\end{equation}
Adding the expressions above, we obtain from (\ref{eqR3})
\begin{equation}
\label{A.5}
R_3 (t) = \int_0^\infty ds e^{-st} \mathcal{R}_3 (s) ds \ ,
\end{equation}
where
\begin{equation}
\label{A.8}
\mathcal{R}_3 (s) =
\left \{ \frac{1}{4 (1+s/2)^{3/2}} - \frac{s}{(s+2)^2 \sqrt{1+s} } 
+ \frac{3 s^2-4}{2 (s+2)^3 (s+1)^{3/2} } \right \}
\end{equation}
We define
\begin{equation}
\label{A.7.0}
s = -1 + \frac{(y-1)^2}{4y}
\end{equation}
Clearly $s$ is increasing in $y$ and maps 
$ \left ( 3 + 2 \sqrt{2}, \infty \right ) $ to
$(0, \infty)$.
Thus
\begin{equation}
\label{A.8.0}
\mathcal{R}_3 (s(y)) =
\frac{4 y^{3/2} (2-\sqrt{2}) (y-3-2\sqrt{2})}{
(y+1)^6 (y-1)^3} P_5 (y) \ ,
\end{equation}
where
\begin{equation}
\label{A.7}
P_5 (y) = 1 - ( 21 - 10 \sqrt{2} ) y + (42-2 \sqrt{2} ) y^2 
- (42 + 2 \sqrt{2} ) y^3 + (10 \sqrt{2} + 21) y^4 - y^5
\end{equation}
Elementary inequalities imply that $P_5$ has
a real root $y_0 \in (33.851, 33.852)  $ 
corresponding to $s=s_0 \in (6.9701, 6.9704)$
By factoring out $y-y_0$ and rexpanding the rest in powers of $y-5$, we find
\begin{equation}
\label{A.8.0.0}
P_5 (y) = (y-y_0) \left ( A_0 + A_1 (y-5) + A_2 (y-5)^2 + A_3 (y-5)^3 - (y-5)^4
\right )  
\end{equation}
where
\begin{equation}
\label{A.9}
A_0 = -y_0^4 + 
(10 \sqrt{2}+16) y_0^3+
(48 \sqrt{2}+38) y_0^2+ (238 \sqrt{2}+232) y_0+1139+1200 \sqrt{2}
\end{equation}
\begin{equation}
\label{A.10}
A_1 = -y_0^3+ (11+10\sqrt{2}) y_0^2+(98 \sqrt{2}+93) y_0+728 \sqrt{2}+697
\end{equation}
\begin{equation}
\label{A.11}
A_2 = -y_0^2+(10 \sqrt{2}+6) y_0+148 \sqrt{2}+123
\end{equation}
\begin{equation}
\label{A.12}
A_3 = 1+10 \sqrt{2} - y_0
\end{equation}
It is readily checked that for $y_0$ in the interval $(33.851, 33.852)$,
the coefficients $A_0$, ..., $A_3$ are negative, implying that 
there has only one zero for $y \ge 3 + 2 \sqrt{2} > 5 $, namely
$y=y_0$.
This immediately implies that for $s \in (0, s_0)$ where $s_0=6.97\cdots$   we have
$\mathcal{R}_3 (s) > 0$, while for $s > s_0$, $\mathcal{R}_3 < 0$. 

We now minimize $\mathcal{R}_3 (s (y))$ for 
$y > y_0 =33.851\cdots$.
By simple estimates of the derivative, $y^{-5} P_5 (y)$ is seen to be 
decreasing, and $0 \ge y^{-5} P_5 > -1$ for $y \in [y_0, \infty)$.      
In this interval $y-3-2\sqrt{2} \le y-1$ and thus
for $s > s_0$,
(or $y > y_0$) we obtain using  (\ref{A.8.0}), that
\begin{multline}
\label{A.14.0}
\mathcal{R}_3 (s(y)) \ge 
-\frac{4 y^{13/2} (2-\sqrt{2})}{
(y+1)^6 (y-1)^2} \ge  
-\frac{4 y_0^{13/2} (2-\sqrt{2})}{
(y_0+1)^6 (y_0-1)^2} 
\ge 
-0.0107, 
\end{multline}
since
${4 y^{13/2} (2-\sqrt{2})}{
(y+1)^{-6} (y-1)^{-2}}$ is decreasing for $y \in (y_0, \infty)$. 
Therefore since for $t \ge T$,
\begin{equation}
\label{A.14}
0 \le \int_0^{s_0} e^{-st} \mathcal{R}_3 (s) ds \le \int_0^{s_0} e^{-sT} 
\mathcal{R}_3 (s) ds \ , 
\end{equation}
\begin{equation}
\label{A.15}
0 \le -\int_{s_0}^\infty e^{-st} \mathcal{R}_3 (s) 
\le -\int_{s_0}^\infty e^{-s T} \mathcal{R}_3 (s) 
\le \frac{0.0107}{T} e^{-s_0 T} \ ,   
\end{equation}
it follows that
\begin{equation}
\label{eqR3m}
R_3 (t) \le R_3 (T) + 
\frac{0.0107}{T} e^{-s_0 T}  
\le R_3 (T) + 1.02 \times 10^{-8} 
=:R_{3,m} \le 0.02057 
\end{equation}

\subsection{Estimating $R_4$}

Using (\ref{14.5})--(\ref{14.5.0}) in  (\ref{eqR4}) we get
\begin{equation}
R_4 (t) = \frac{1}{8t} (J_1 - I_1 ) \left ( 1 -I_1 \right ) - \frac{I_1}{16 t^2} ( J_1 - I_1^2 ) 
+ \frac{I_0^3}{4} \, 
\end{equation}
where
\begin{equation}
I_1 (t) = 2 t I_0 (t) = t \int_0^\infty \frac{e^{-st}}{(1+s)^{3/2}} ds \in (0, 1) 
\end{equation} 
\begin{equation}
J_1 (t) = 4 t J_0 (t) = t \int_0^\infty \frac{e^{-st}}{(1+s/2)^{3/2}} ds \in (0, 1)
\end{equation}
Clearly $J_1 (t) \ge I_1 (t)$.
This implies that
\begin{equation}
-\frac{1}{16 t^2} \left ( J_1 - I_1^2 \right )  \le 
R_4 (t) \le \frac{1}{8t} (J_1 - I_1) + \frac{I_0^3}{4}
\end{equation}
Now, 
\begin{equation}
\frac{1}{t} ( J_1 -I_1 ) = \int_0^\infty \frac{(1+s)^{3/2} - (1+s/2)^{3/2}}{(1+s/2)^{3/2} (1+s)^{3/2} } e^{-st}     
ds
\end{equation}
which is clearly decreasing in $t$ as is $I_0$ (since they are Laplace
transforms of positive functions),  and therefore they 
attain their  maximum at $t=T \ge 1.99$.
Furthermore, it can be checked that
\begin{multline}
\frac{1}{16 t^2} \left (J_1 - I_1^2 \right ) = \frac{1}{8 t} \int_0^\infty e^{-st} 
\\
\left \{  {\frac {16 s \sqrt {1+s}+10 s^2 \sqrt {1+s} +2 s^3 \sqrt {1+s} 
+ 3 s^2 \sqrt {4+2\,s} + 8 \sqrt{1+s}-8\sqrt {1+s/2} }{ \left( 
2+s \right) ^{3}\sqrt {4+2\,s} \left( 1+s \right) ^{3/2}}}
\right \} ds    
\end{multline}
Arguing in the same way, we see that 
$\frac{1}{16 t^2} \left (J_1 - I_1^2 \right )$ is decreasing and
attains its maximum at $t=T$, as does $\frac{I_0^3}{4}$.
Evaluating $I_0$ and $J_0$ (see, once more,  \cite{Abramowitz}, 7.1.28) we get
\begin{equation}
0 \le \frac{1}{8 t} ( J_1 (t) -I_1 (t) ) +\frac{I_0^3 (t)}{4}
\le \frac{1}{8 T} ( J_1 (T) -I_1 (T) ) +\frac{I_0^3 (T)}{4}
\le \frac{1}{2} J_0 (T) - \frac{1}{4} I_0 (T)  
+ \frac{I_0^3 (T)}{4} \le 0.009042    
\end{equation}
and 
\begin{equation}
0 \le \frac{1}{16 t^2} \left ( J_1 - I_1^2 \right ) 
\le \frac{1}{16 T^2} \left ( J_1 (T) - I_1^2 (T) \right ) 
= \frac{J_0 (T)}{4 T} - \frac{I_0^2 (T)}{4}  \le 0.00572
\end{equation} 
and therefore, for $t \ge T \ge 1.99$, we have 
\begin{equation}
\label{eqR4m}
\Big | R_4 (t) \Big | \le R_{4,m} \le 0.009042  
\end{equation}

\section{Acknowledgments} The work of O.C. and S.T  was partially supported by the NSF
grant DMS 1108794.

\vfill \eject
\end{document}